\documentclass[a4paper, 11pt, reqno]{amsart}
\usepackage{a4wide, amssymb, amsmath, amsthm, latexsym, bbm}

\usepackage[arrow, matrix, curve]{xy}

\usepackage{dsfont}
\usepackage[T1]{fontenc}

\usepackage[active]{srcltx}

\usepackage{hyperref}
    
\allowdisplaybreaks

\numberwithin{equation}{section}

\DeclareMathOperator{\1I}{\mathbbm{1}}
\DeclareMathOperator{\aco}{aco}

\DeclareMathOperator{\id}{id}
\DeclareMathOperator{\Id}{Id}

\DeclareMathOperator{\RsL}{R_sL}
\DeclareMathOperator{\sgn}{sgn}
\DeclareMathOperator{\supp}{supp}

\newcommand{\N}{\mathbb N}
\newcommand{\Z}{\mathbb Z}

\newcommand{\R}{\mathbb R}
\newcommand{\C}{\mathbb C}
\newcommand{\K}{\mathbb K}

\newcommand{\mE}{\mathcal{E}}
\newcommand{\mF}{\mathcal{F}}

\newcommand{\mL}{\mathcal{L}}
\newcommand{\mM}{\mathcal{M}}

\newcommand{\mR}{\mathcal{R}}
\newcommand{\mS}{\mathcal{S}}
\newcommand{\mT}{\mathcal{T}}

\newcommand{\abs}[1]{\lvert#1\rvert}
\newcommand{\BB}{\mathfrak{B}}
\newcommand{\ds}{\displaystyle}
\newcommand{\del}{\partial}
\newcommand{\E}{{\mathbb E}}
\newcommand{\eps}{\varepsilon}

\newcommand{\HU}{$H^\infty$}
\newcommand{\imp}{\Rightarrow}

\newcommand{\into}{\hookrightarrow}
\newcommand{\iso}{\cong}

\newcommand{\lsk}{\langle}
\newcommand{\la}{\lambda}
\newcommand{\mal}{\cdot}

\newcommand{\ohne}{\backslash}
\newcommand{\om}{\supseteq}
\newcommand{\ph}{\varphi}
\newcommand{\rechts}{\longrightarrow}
\newcommand{\rsk}{\rangle}
\newcommand{\tensor}{\otimes}

\newcommand{\tm}{\subseteq}

\newcommand{\Xdot}{\dot{X}}

\theoremstyle{plain}
  \newtheorem{satz}{Proposition}[section]
  \newtheorem{prop}[satz]{Proposition}
  \newtheorem{lemma}[satz]{Lemma}
  
  \newtheorem{cor}[satz]{Corollary}
  \newtheorem{theorem}[satz]{Theorem}
  \newtheorem{bem}[satz]{Remark}
  \newtheorem{remark}[satz]{Remark}
\theoremstyle{definition}
  \newtheorem{definition}[satz]{Definition}
  \newtheorem{defsatz}[satz]{Definition/Proposition}
  \newtheorem{example}[satz]{Example}
  \newtheorem{bspe}[satz]{Examples}

\title{$\mR_s$-sectorial operators and generalized Triebel-Lizorkin spaces}

\author{Peer Kunstmann}
\author{Alexander Ullmann}

\address{Department of Mathematics,
Karlsruhe Institute of Technology, 76128 Karlsruhe, Germany.}
\email{peer.kunstmann@kit.edu}
\email{ullmann@kit.edu}

\thanks{This work was partially supported by the Deutsche Forschungsgemeinschaft (DFG) (We 2847/1-2).}



\begin{document}

\begin{abstract}
We introduce a notion of generalized Triebel-Lizorkin spaces associated with sectorial operators in Banach function spaces. Our approach is based on holomorphic functional calculus techniques. Using the concept of $\mR_s$-sectorial operators, which in turn is based on the notion of $\mR_s$-bounded sets of operators introduced by Lutz Weis, we obtain a neat theory including equivalence of various norms and a precise description of real and complex interpolation spaces. Another main result of
this article is that an $\mR_s$-sectorial operator always has a bounded \HU-functional calculus in its associated generalized Triebel-Lizorkin spaces.
\end{abstract}

\maketitle

\tableofcontents

\section{Introduction}

The theory of function spaces is a wide and important topic and has applications in various fields of analysis, for example in regularity theory for partial differential equations, cf. e.g. \cite{triebel-interpolation}, \cite{amann}. It is well known that many classical function spaces as Sobolev spaces, Bessel-potential spaces or H\"older-Zygmund spaces, can be subsumed in the general notion of Besov and Triebel-Lizorkin spaces, cf. e.g. \cite{triebel-1}, \cite{triebel-2}.
The most common definition of these spaces is based on the Fourier transform and a dyadic decomposition of the Fourier image. From an operator theoretic point of view, these spaces are closely related to the Laplace operator. In fact, Besov spaces are real interpolation spaces of the domains of powers of the Laplacian. For a sectorial operator $A$ with domain $D(A)$ in a Banach space $X$, an equivalent norm for the real interpolation space $(X,D(A))_{\theta,s}$ ($\theta\in (0,1), s\in [1,\infty]$) is given using functional calculus as
\begin{equation} \label{intro_Besov}
\|x\|_{(X,D(A))_{\theta,s}} \approx \|x\|_X + \bigg( \int_0^\infty \|t^{-\theta} \ph(tA)x\|_X^s \, \frac{dt}{t}\bigg)^{1/s}
\end{equation}

(with the usual modification if $s=\infty$), where $\ph\neq 0$ is a suitable bounded holomorphic function. This goes back to Komatsu, cf. \cite{komatsu2}, and a comprehensive treatment is given in \cite{haase}.\\

The norm in (classical) Triebel-Lizorkin spaces $F^{2\theta}_{p,s}$ can be rewritten in a similar way as
\begin{equation} \label{intro_TL}
\|x\|_X + \bigg\| \bigg( \int_0^\infty |t^{-\theta} \ph(tA)x|^s \, \frac{dt}{t}\bigg)^{1/s}\bigg\|_X,
\end{equation}

where $X=L^p(\R^n)$, $A=-\Delta$ and $p\in [1,\infty), s\in [1,\infty], \theta\in (0,1)$, if e.g. $\ph(z)=ze^{-z}$, cf. \cite{triebel-artikel82}.\\

The norm expression in (\ref{intro_TL}) still makes sense for a general sectorial operator in a Banach function space $X$, and it seems reasonable to use (\ref{intro_TL}) to define generalized Triebel-Lizorkin spaces for a general sectorial operator $A$. This is what we shall do in this paper.\\

For Schr\"odinger operators $H=-\Delta+V$ with certain potentials $V$, generalized Triebel-Lizorkin spaces have been studied e.g. in \cite{zheng1}, \cite{zheng2}, via an approach that is much closer to the usual definition. The two (closely interrelated) main differences to the present approach are:
\begin{itemize}\itemsep0pt
\item the operator $H$ is \emph{self-adjoint} in $L^2(\R^n)$ whereas we
  allow for general sectorial operators $A$ in a Banach function  space,
\item the functional calculus of the spectral theorem for $H$ is used,
  which contains functions with compact support, whereas we resort to
  the holomorphic functional calculus for sectorial operators and thus have to
  use holomorphic functions $\ph$ in (\ref{intro_TL}).
\end{itemize}

Another issue is that the usual Littlewood-Paley-like definition of Besov- and Triebel-Lizorkin spaces (and also \cite{zheng1}, \cite{zheng2}) uses discrete dyadic sums rather than continuous integral expressions like (\ref{intro_Besov}), (\ref{intro_TL}). For $\ph(z)=ze^{-z}$, discrete dyadic analogs of (\ref{intro_Besov}) can be extracted from \cite{triebel-interpolation} (based on the notion of quasi-linearizable interpolation couples and properties of the $K$-functional). In the context of $H^\infty$-calculus, there is no general study of such expressions. We address the question of discrete analogs of (\ref{intro_TL}) in Subsection 3.3. It turns out that one has to be much more careful in the choice of $\ph$. The condition we give below is certainly a bit technical in nature. On the other hand, our study also yields new discrete analogs for expressions (\ref{intro_Besov}), cf. Remarks \ref{bem_diskrete-Normen-Haase}, \ref{bem_diskrete-Normen-Haase-im-Besovraum}.\\

Let us mention right away that the case $s=2$ is special. Much more is known and much more can be proved. This is, on the one hand, due to the ubiquity of square function expressions in harmonic analysis, which were also among the first topics studied in the theory of $H^\infty$-functional calculus, cf. e.g. the fundamental paper \cite{mcintosh-1986} for the Hilbert space case (where actually the expressions (\ref{intro_Besov}) and (\ref{intro_TL}) coincide) and \cite{CDMY}, where more generally $X=L^p$ is covered. On the other hand, Khintchine's inequality allows to rephrase square function expressions in terms of random sums involving Rademacher or Gaussian sequences. These, however, make sense in arbitrary Banach spaces $X$, and the relations to $H^\infty$-functional calculus have been established in \cite{kaltonweis2}, \cite{KKW06}, \cite{kaltonweisAA}. Moreover, in \cite{kaltonweisAA}, \cite{suarez-weis}, \cite{KKW06}, interpolation methods have been introduced (based on Rademacher and Gaussian random sums) that can be used to define intermediate spaces between $X$ and the domain $D(A)$ of a sectorial operator $A$. For
$X=L^p(\R^n)$, $1<p<\infty$, and $A=-\Delta$ these intermediate spaces are the Bessel potential spaces. The main new feature in the present paper thus is the study of the case $s\neq 2$. In order to give a  meaning to (\ref{intro_TL}) we cannot work in a general Banach space $X$ and restrict to Banach function spaces. Moreover, there is no interpolation method behind (even if $X$ is a Banach function space, the domain $D(A)$ of a sectorial operator $A$ in $X$ is not a Banach
function space, in general). Nevertheless, we shall use expressions like (\ref{intro_TL}) to establish a theory of  \emph{$s$-intermediate   spaces} between $X$ and $D(A)$ which, for $X=L^p(\R^n)$, $1<p<\infty$, and $A=-\Delta$ coincide with classical Triebel-Lizorkin spaces $F^\alpha_{p,s}(\R^n)$.\\

Let us now be a bit more specific about the content of the present paper. As in the study of Besov- and Triebel-Lizorkin spaces, a central issue in handling the norms (\ref{intro_Besov}), (\ref{intro_TL}) is that they are independent of the special choice of (suitable) auxiliary functions $\ph$. For Besov-spaces (or in the general case real interpolation spaces) this can be ensured by sectoriality of the operator $A$. In the case of the generalized Triebel-Lizorkin norms (\ref{intro_TL}), things are more subtle, since we need in addition control of the norm of the pointwise expressions in (\ref{intro_TL}). For this we introduce the concept of $\mR_s$-sectorial operators, based on the notion of $\mR_s$-boundedness of operators introduced in \cite{Weis01}: If $X$ is a Banach function space, a set $\mT\tm L(X)$  is called \emph{$\mR_s$-bounded} if there is a constant $C>0$ such that for all $n\in\N$, $T\in\mT^n$ and $x\in X^n$ we have an estimate
\begin{equation}\label{intro-Rs-bdd}
\Big\| \Big(\sum_{j=1}^n |T_jx_j|^s\Big)^{1/s} \Big\|_X \le C\mal \Big\| \Big(\sum_{j=1}^n |x_j|^s\Big)^{1/s} \Big\|_X
\end{equation}

(with the usual modification if $s=\infty$). An operator $A$ in $X$ is called \emph{$\mR_s$-sectorial} if there is an open sector $\Sigma\subsetneq\C$ such that $\{zR(z,A) \,|\, z\in\Sigma\}$ is $\mR_s$-bounded. In Subsection \ref{subsection_Equivalence-of-s-power-function-norms} we will show that in this case the $s$-power function norm expressions in (\ref{intro_TL}) are equivalent for a certain class of auxiliary functions $\ph$ using arguments from \cite{LeM04}. This is the central
technical tool to deal with these norms. Of course, estimates like (\ref{intro-Rs-bdd}) have been studied before in classical harmonic analysis where they appear in the context of vector-valued extensions of singular integral operators (cf. below).\\

Using these concepts, we can use (\ref{intro_TL}) to define generalized Triebel-Lizorkin spaces for $\mR_s$-sectorial operators and conclude that these spaces are reasonable intermediate spaces. In particular, they are well-behaved under real and complex interpolation. Another main result of this work is that the operator $A$ always has a bounded \HU-calculus in its associated generalized Triebel-Lizorkin spaces.\\

In \cite{artikel-Rs-Hoo} we have shown that large classes of differential operators have an $\mR_s$-bounded \HU-calculus, i.e. the set $\{ f(A) \,|\, f\in H^\infty(\Sigma), \|f\|_{\infty,\Sigma} \le 1\}$ is $\mR_s$-bounded for some open sector $\Sigma\subsetneq\C$. In particular, these operators are $\mR_s$-sectorial. Therefore, the theory of generalized Triebel-Lizorkin spaces developed in the present article can be applied to those operators. Moreover, in \cite{artikel-TL-for-PDO} we shall show that, for several classes of elliptic differential operators, generalized Triebel-Lizorkin spaces coincide with classical Triebel-Lizorkin spaces. On the one hand, this gives new representations of the spaces $F^\theta_{pq}$, on the other hand, by the result mentioned above this means that these operators have a bounded \HU-calculus in the (classical) Triebel-Lizorkin spaces $F^\theta_{pq}$.\\

This paper is organized as follows: We start with some preliminaries in Section \ref{section_prelim}. Beside standard facts about Banach function spaces and the \HU-calculus for sectorial operators, we give a detailed exposure of the concept of $\mR_s$-boundedness in Banach function spaces in Subsection \ref{subsection-Rs-bd}. The explicit notion of $\mR_s$-boundedness in $L^p$-spaces was introduced by Lutz Weis in \cite{Weis01}, and it has been used e.g. in \cite{Weis01}, \cite{BK02} to deduce $\mR$-boundedness of semigroups of operators and hence maximal regularity of the related Cauchy problem. Based on this concept we study the notion of $\mR_s$-sectorial operators and $\mR_s$-bounded \HU-calculus in Section \ref{section_Rs-Sekt.}. In Subsection \ref{subsection_Equivalence-of-s-power-function-norms} we show the equivalence of (continuous) $s$-power function norms (\ref{intro_TL}) for certain auxiliary functions $\ph$, and in Subsection \ref{subsection_Discrete_s-power_function_norms} we study in detail discrete counterparts of (\ref{intro_TL}) of the form
\begin{equation} \label{intro_TL-discrete}
\|x\|_X + \Big\| \Big( \sum_{j\in\Z} |2^{-j\theta} \ph(2^jA)x|^s \Big)^{1/s} \Big\|_X,
\end{equation}

and show that also these norms are equivalent to (\ref{intro_TL}), where one has to impose more restrictions on the auxiliary function $\ph$. This will also lead to new results for discrete counterparts of the norms (\ref{intro_Besov}) for real interpolation spaces.\\

In Section \ref{section_The-associated-s-intermediate-spaces} we turn to the central topic of this paper. We start with the definition of the associated homogeneous and inhomogeneous \emph{$s$-intermediate   spaces} $\Xdot^\theta_{s,A}, X^\theta_{s,A}$ which we will also refer to as generalized Triebel-Lizorkin spaces. This will be justified in Proposition \ref{satz_s-Raum-von-Laplace=Fspq}, where we show that the $s$-intermediate spaces for the Laplace operator $A=-\Delta$ coincide with the classical Triebel-Lizorkin spaces. After the definition and some elementary properties in Subsection \ref{subsection_Element-prop-of-s-spaces} we will show in Subsection \ref{subsection_The-s-spaces-as-intermediate-spaces-and-interpolation} that the $s$-intermediate spaces are indeed intermediate spaces for the interpolation couple $(X, D(A^m))$, and we will explore real and complex interpolation of these spaces. In the last Subsection \ref{subsection_The-part-of-A-in-the-s-intermediate-spaces} we will present one of the main results of this work and show that the ``part'' of $A$ (which has to be defined properly) always has a bounded \HU-calculus in its scale of homogeneous $s$-intermediate spaces $\Xdot^\theta_{s,A}, \theta\in\R$. In the case that $A$ is
invertible or that $A$ has a bounded \HU-calculus in $X$, it also has a bounded \HU-calculus in the inhomogeneous spaces $X^\theta_{s,A}$ if $\theta>0$. We thus establish an analog of Dore's Theorem that states that an invertible sectorial operator $A$ in a Banach space $X$ has a bounded \HU-calculus in the scale of real interpolation spaces $(X,D(A))_{\theta,p}$ for $p\in[1,\infty], \theta\in (0,1)$, cf. \cite{dore-interpolation1}, \cite{dore-interpolation2}. The case $s=2$ is known in principle (\cite{LeM04}, \cite{LeMerdy-square-fcts}, \cite{kaltonweisAA}, \cite{KKW06}, \cite{kriegler}), but the case $s\neq 2$ is entirely new. We shall give applications of this result in \cite{artikel-TL-for-PDO}.\\

It is a reasonable claim that the generalized Triebel-Lizorkin spaces introduced in this paper coincide with the corresponding spaces for Schr\"odinger operators $H=-\Delta+V$ defined in \cite{zheng1}, \cite{zheng2}, but we will not elaborate on this topic in this paper. A few thougts are given in the concluding remarks in Section \ref{section_Concluding_remarks}.\\

\section{Preliminaries} \label{section_prelim}

\subsection{The \HU-calculus for sectorial operators} \label{subsection-Hoo}

We shall be brief and refer to the standard literature, e.g. \cite{haase} or \cite{levico}, for details. A linear operator $A:X\om D(A)\to X$, where $X$ is a complex Banach space, is called a \emph{sectorial operator of type $\omega \in [0,\pi)$} if the spectrum $\sigma(A)$ is contained in the sector $\overline{\Sigma}_\omega$, where $\Sigma_\sigma := \{ z\in\C\ohne (-\infty,0] \;|\;\, |\arg(z)| < \sigma\}$ if $\sigma\in (0,\pi]$ and $\Sigma_0 := (0,\infty)$, and $\sup_{z \notin \overline{\Sigma}_{\sigma}} \|zR(z,A)\| < \infty$ for all $\sigma\in (\omega,\pi)$. The infimum $\omega(A)$ over all such $\omega\in (0,\pi/2)$ is called {\em the type of $A$}. In this article, we will additionally always assume that $A$ is injective with dense domain and range (Observe that, by the sectoriality condition,
density of $R(A)$ already implies that $A$ is injective, and if $X$ is reflexive, then $D(A)$ is always dense).\\

For $f:\Sigma_\sigma\to\C$ let $\|f\|_{\infty,\sigma}:= \sup_{z\in\Sigma_\sigma}|f(z)|$, where we sometimes drop the index $\sigma$ in notation if there is no risk of confusion. We introduce the algebra $H^\infty(\Sigma_\sigma):= \{ f:\Sigma_\sigma\to \C \,|\, f$ analytic, $\|f\|_{\infty,\sigma} < \infty \}$ of bounded analytic functions on the sector $\Sigma_\sigma$ and the subalgebra $H_0^\infty(\Sigma_\sigma)$ consisting of those $f\in H^\infty(\Sigma_\sigma)$ for which there exists an $\eps>0$ with $\sup_{z\in\Sigma_\sigma} \big( (|z|^\eps  \vee |z|^{-\eps})  |f(z)| \big) < \infty$. Let $\omega\in (\omega(A),\sigma)$ and define the path of integration $\Gamma_\omega(t) := |t|e^{-\sgn(t)i\omega}$ for all $t\in\R$, then
\begin{equation}
\ph \mapsto \ph(A):=\frac{1}{2\pi i} \int_{\Gamma_\omega} \ph(\la)R(\la,A)\, d\la
\end{equation}

defines an algebra homomorphism $H_0^\infty(\Sigma_{\sigma})\to L(X)$ that is independent of $\omega \in (\omega(A),\sigma)$ and only depends on the germ of $\ph$ on $\overline{\Sigma}_{\omega(A)}$. By a standard extension procedure we obtain a functional calculus for all $f\in H^\infty(\Sigma_\sigma)$ and even for a larger class of holomorphic functions: Let $\rho(\la) := \la(1+\la)^{-2}, \la\notin (-\infty,0)$ and $\BB(\Sigma_\sigma) := \{ f:\Sigma_\sigma \to\C \:|\: z\mapsto \rho(z)^m f(z) \in H_0^\infty(\Sigma_\sigma)\text{ for some } m\in\N\}$, then it is easy to check that $\rho^m(A)=\rho(A)^m=A^m(1+A)^{-2m}$, and $\rho(A)$ is an injective operator. Let $f \in \BB(\Sigma_\sigma)$ and choose $m\in\N$ such that $\rho^m f\in H_0^\infty(\Sigma_\sigma)$, then the operator $(\rho^m f)(A)\in L(X)$ is well defined and we can define $f(A) := (\rho(A))^{-m} (\rho^m f)(A)$. It can be shown that the definition of $f(A)$ is independent of $m\in\N$ such that $\rho^m f\in H_0^\infty(\Sigma_\sigma)$, and that $f\mapsto f(A)$ is an (abstract) functional calculus for $A$ in the sense of \cite{haase}, Chapter 1.3.\\

Let $\sigma\in (\omega(A),\pi]$, then $A$ is said to have a {\em bounded $H^\infty(\Sigma_\sigma)$-calculus} if
\[
M^\infty_{\sigma}(A) := \sup\{ \|f(A)\| \:|\: f\in H^\infty(\Sigma_\sigma), \|f\|_{\infty,\sigma}\le 1\} < \infty.
\]

In this case, $\omega_{H^\infty}(A) := \inf\{ \sigma \in (\omega(A),\pi] \:|\:M^\infty_{\sigma}(A) < \infty \}$ is called the {\em \HU-type} of $A$, and we will also just say that $A$ has a bounded $H^\infty$-calculus.\\

It is an easy consequence of the so-called convergence lemma and the closed graph theorem, that $A$ has a bounded $H^\infty(\Sigma_\sigma)$-calculus if and only if there is a $C>0$ such that
\[
\|\ph(A)\| \le C\, \|\ph\|_{\infty,\sigma} \quad\text{ for all } \ph\in H_0^\infty(\Sigma_\sigma),
\]

and in this case $M^\infty_{\sigma}(A)\le C$,  cf. e.g. \cite{levico}, Remark 9.11 and \cite{haase}, Proposition 5.3.4.\\

As a special class of analytic functions $f$ that still yield a nice representation formula and ensure that $f(A)$ is bounded we introduce the {\em extended Dunford-Riesz class}, which is defined by $\mathcal{E}(\Sigma_\sigma) := H_0^\infty(\Sigma_\sigma) \oplus \Big\lsk \frac{1}{1+\id_{\Sigma_\sigma}}\Big\rsk_\C \oplus \lsk \1I_{\Sigma_\sigma}\rsk_\C$ and $\mathcal{E}_{\omega} := \bigcup_{\sigma'>\omega} \mathcal{E} (\Sigma_{\sigma'})$ for all $\omega\in [0,\pi)$. Then $\mathcal{E}(\Sigma_\sigma) $ is the algebra of bounded analytic functions on $\Sigma_\sigma$ that have finite polynomial limits in $0$ and $\infty$. Here we say that $f$ has a finite polynomial limit in $0$, if there is an $a\in\C$ and $\alpha>0$ such that $f(z)-a= O(|z|^\alpha)$ as $z\to 0$, and a finite polynomial
limit in $\infty$, if the latter is true for $f(z^{-1})$. In this case, the values $f(0), f(\infty)\in\C$ are well defined. Moreover, by the mean value theorem, bounded holomorphic functions on $\Sigma_\sigma$ that are decaying to $0$ at $0$ and $\infty$, or that are holomorphic in a neighborhood of $0$ and $\infty$, respectively, belong to the class $\mathcal{E}(\Sigma_\sigma)$. For $f\in\mathcal{E}(\Sigma_\sigma)$ let $\ph:=  f - \frac{f(0)-f(\infty)}{1+\id_{\Sigma_\sigma}} - f(\infty)\,\1I_{\Sigma_\sigma}$ be the corresponding $H_0^\infty$-function, then it is easily checked that
\[
f(A) = \ph(A) + (f(0)-f(\infty))\,(1+A)^{-1} + f(\infty)\,\id_X.
\]

For details we refer to \cite{haase}, Section 2.2.\\

\subsection{Banach function spaces} \label{subsection_B.f.s.}

For this subsection we refer to the standard references \cite{bennett} Chapter 1 and \cite{zaanen} Chapter 15. Let $(\Omega,\mu)$ be a $\sigma$-finite measure space. We fix a \textit{$\mu$-localizing sequence} $(\Omega_n)_{n\in\N}$, i.e. an increasing sequence of $\mu$-measurable subsets such that $\mu(\Omega_n)<\infty$ for all $n\in\N$ and $\bigcup_{n\in\N} \Omega_n=\Omega$. A $\mu$-measurable subset $M\tm \Omega$ will be called \emph{$(\Omega_n)_{n\in\N}$-bounded} if $M\tm \Omega_n$ for some $n\in\N$. We will use the terminology that a property for a $\mu$-measurable function $f$ on $\Omega$ holds \emph{$(\Omega_n)_{n\in\N}$-locally} if it holds for $f|_M$ for all $(\Omega_n)_{n\in\N}$-bounded sets $M$. In particular we introduce the following notation:\\

If $f_n:\Omega \to \K, n\in\N_0$ are $\mu$-measurable, we say that $f_n\to f_0$ \textit{converges $(\Omega_n)_{n\in\N}$-locally in measure} for $n\to\infty$ if $f_n|_M \to f_0|_M$ in measure for $n\to\infty$ for all  $(\Omega_n)_{n\in\N}$-bounded sets $M$. We denote by $M^*(\Omega,\mu):=M^*(\mu)$ the class of $\mu$-measurable extended scalar-valued functions on $\Omega$, by $M(\Omega,\mu):=M(\mu)$ the space of $\mu$-measurable scalar-valued functions on $\Omega$, endowed with the topology of $(\Omega_n)_{n\in\N}$-local convergence in measure, and by $M^+(\Omega,\mu):=M^+(\mu)$ the cone of $\mu$-measurable functions on $\Omega$ with values in $[0,\infty]$, where as usual we identify functions which are pointwise equal outside a $\mu$-nullset. Moreover we define the spaces $L^1_{loc}(\Omega,\mu) := \{ f\in M(\mu)  \;|\; f|_M \in L^1(M)$ for all $(\Omega_n)_{n\in\N}$-bounded sets $M\}$ of locally integrable functions, endowed with the topology of convergence on $(\Omega_n)_{n\in\N}$-bounded sets. Finally let $S(\Omega,\mu) := \{ f\in M(\mu) \;|\; f(\Omega) $ is finite and $\supp(f)$ is $(\Omega_n)_{n\in\N}$-bounded$\}$ be the space of step functions with bounded support. In general, all these spaces depend on the special choice of the underlying $\mu$-localizing sequence $(\Omega_n)_{n\in\N}$. Nevertheless we will suppress the explicit notation of the sequence $(\Omega_n)_{n\in\N}$ in the sequel but keep it in mind.

\begin{definition}
A map $\rho: M^+(\mu)\to [0,\infty]$ is called a {\em Banach function norm}, if for all $f,g,f_n\in M^+(\mu) (n\in\N), \alpha>0$ and $M\tm \Omega$ $\mu$-measurable the following properties hold:
\begin{itemize}
\item[(B1)] $\rho(f)=0 \iff f=0$ $\mu$-a.e. , $\rho(\alpha f)=\alpha\rho(f)$ and $\rho(f+g)\le\rho(f)+\rho(g)$ ({\em norm properties}),
\item[(B2)] $0\le g\le f$ $\mu$-a.e. $\imp \rho(g)\le\rho(f)$ ({\em monotonicity}),
\item[(B3)] $0\le f_n \nearrow f$ $\mu$-a.e. $\imp \rho(f_n) \nearrow \rho(f)$ ({\em Fatou property}),
\item[(B4)] $M$ bounded $\imp \rho(\1I_M)< \infty$,
\item[(B5)] $M$ bounded $\imp \int_M f\,d\mu \le C_M \rho(f)$, where $C_M>0$ is a constant independent of $f$.
\end{itemize}

In this case, $X:=X(\rho):= \{ f \,|\, f\in M^*(\mu), \rho(|f|) < \infty \}$, endowed with the norm $f \mapsto \|f\|_X := \rho(|f|)$, is called a {\em Banach function space}.
\end{definition}

By (B5) we have $X(\rho)\into L^1_{loc}(\Omega,\mu)$, and (B3) implies that appropriate versions of the classical Fatou Lemma and monotone convergence theorem hold in $X$, for a detailed exposition cf. \cite{bennett}, Chapter 1.1.\\

Standard examples for Banach function spaces are $L^p$-spaces, Lorentz spaces and Orlicz spaces. Observe that our definition of a Banach function spaces is a little more general than the one given in \cite{bennett} since conditions (B4), (B5) are only required to hold on $(\Omega_n)_{n\in\N}$-bounded sets in the sense as discussed at the beginning. Hence, our notion of Banach function spaces also covers mixed spaces $L^pL^q$. In \cite{bennett} conditions (B4), (B5) are formulated for the collections of all $\mu$-measurable sets $M$ of finite measure. Nevertheless, the proofs given in \cite{bennett} also work in our situation, hence we will usually cite \cite{bennett} as our standard reference. In \cite{zaanen} a more general notion of Banach function spaces, which are called K\"othe spaces there, are considered.\\

For a Banach function space $X$ we have $S(\Omega,\mu)\into X\into M(\mu)$, where the second inclusion is continuous (\cite{bennett}, Theorem 1.4). This implies e.g. that each convergent sequence $f\in X^\N$ contains a subsequence that converges $\mu$-a.e., cf. \cite{bennett}, Theorem I.1.7 (vi). One can also show that $X$ is a complete lattice, to be more precise:
\begin{satz} \label{satz-X-als-Verband-vollst}
Let $X$ be a Banach function space. Then $X$ is a complete sub-lattice of $M^*(\mu)$, and to every $F\tm X$ there is a countable subset $F_0\in F$ such that $\sup F = \sup F_0$.
\end{satz}

This can be proven in the same manner as in the case $X=L^p$, cf. e.g. \cite{dunford} Cor. IV.11.7.\\

We will usually also need the following additional property:
\begin{itemize}
\item[(B6)] If $f\in X$ and $(M_n)_{n\in\N}$ is a decreasing sequence of $\mu$-measurable sets with $\1I_{M_n}\to 0$ $\mu$-a.e., then $\|f \1I_{M_n}\|_X \to 0$ for $n\to\infty$ ({\em absolute continuity}).
\end{itemize}

A Banach function space that fulfills (B1)-(B6) will be called a Banach function space with \emph{absolute continuous norm}. Observe that (B6) implies that the space of step functions $S(\Omega,\mu)$ is dense in $X$, cf. e.g.\cite{bennett}, Theorem I.3.11., and (B6) is equivalent to the \emph{$\sigma$-order-continuity} of the lattice $X$ as well as to the validity of Lebesgue's theorem, compare \cite{bennett} Prop. 3.5, 3.6.\\

If $E\neq\{0\}$ is a Banach space, we define the $E$-valued extension of the function space $X$ as $X(E) := \{ F\in M(\Omega,E) \:|\:  |F|_E \in X \}$, where $|F|_E := \| F(\mal)\|_E$. Letting $\|F\|_{X(E)} := \| \, |F|_E\,\|_X$ makes $X(E)$ a Banach space (cf. \cite{calderon}). If in addition $X$ has absolute continuous norm, then the space of step functions $S(\Omega,\mu)$ is dense in $X$, hence $X\tensor E$ is dense in $X(E)$ in this case.\\

Observe that the construction of the space $X(E)$ works for real and complex Banach spaces $E$ and thus makes $X(E)$ to a real or complex Banach space, respectively. In particular, the space $X(\C)$ is well-defined and will be called a {\em complex Banach function space}. In the sequel we will also just write $X$ for a complex Banach function space having in mind that $X = \widetilde{X}(\C)$ for some (real) Banach function space $\widetilde{X}$. In this case,  properties as (B6) for $X$ are always understood as $\widetilde{X}$ having this property, and $X(E)$ denotes the space $ \widetilde{X}(E)$ for any Banach space $E$.\\

We note the following proposition for $X(E)$-valued integrable functions, which is well known if $X=L^p(\Omega,\mu)$, cf. \cite{dunford}, Chapter III.11. Since $X$ locally embeds into $L^1$ the proof given there easily carries over to the more general situation considered here, :
\begin{satz} \label{int_mit_werten_in_X}
Assume that $X$ has absolute continuous norm. Let $(J,\nu)$ be a $\sigma$-finite measure-space and $F:J\to X(E)$ be an integrable function. Then there exists a $\nu\tensor\mu$-measurable function $f:J\times\Omega\to E$ such that:
\begin{enumerate}
\item $f(t,\mal)$ is a representative of $F(t)$ for $\nu$-a.e. $t\in J$,
\item $f(\mal,\omega)$ is integrable for $\mu$-a.e. $\omega\in\Omega$,
\item the mapping $\omega \mapsto \int_J f(t,\omega)\,d\nu(t)$ is a representative of $\int_J F(t)\,d\nu(t)$.
\end{enumerate}
\qed\end{satz}

In the situation of Proposition \ref{int_mit_werten_in_X} we obtain in particular
\begin{equation} \label{ungl-B.f.s.-integral-betrag}
\Big|\int_J F(t) \,d\nu(t)\Big|_E \le \int_J |f(t,\mal)|_E\,d\nu(t) = \int_J |F(t)|_E \,d\nu(t).
\end{equation}

Finally we mention that if $U\tm\C$ is open and $F:U\to X(E)$ is analytic, one can choose a version of $F$ with analytic paths, i.e., there is a measurable function $f:U\times \Omega \to E$ such that $f(\mal,w)$ is analytic for a.e. $w\in\Omega$, and for all $z\in U, k\in\N^0$, the function $\del_z^k f(z,\mal)$ is a representative of $F^{(k)}(z)$. This goes back to Stein, cf. \cite{stein}, III.2 Lemma, a detailed exposition for $X=L^p$ can be found in \cite{desch-homan}, and a variant of the proof given there also works in our situation.\\

\subsection{$\mR_s$-boundedness} \label{subsection-Rs-bd}

In this subsection we give a short exposition of the notion of $\mR_s$-boundedness in Banach function spaces, which was introduced for $L^p$-spaces bei Lutz Weis in \cite{Weis01}. In fact, the concept of $\mR_s$-boundedness in $L^p$-spaces, although not denoted in this way, is a subject of classical harmonic analysis, mainly considered in the framework of vector-valued singular integrals, cf. e.g. the monographs \cite{stein-dick}, \cite{garcia} or \cite{grafakos}. In \cite{Weis01} and in the sequel e.g. in \cite{BK02}, the notion of $\mR_s$-boundedness was a central tool to show maximal regularity of certain sectorial operators. Crucial in this context is the coincidence of $\mR_2$-boundedness in $L^p$-spaces with $\mR$-boundedness if $1\le p < \infty$, which in turn is a central tool in dealing with the question of maximal regularity. In fact, many of the assertions we present in this section are already indicated in \cite{Weis01}, or they are variants of corresponding assertions for $\mR$-boundedness as shown in \cite{levico}, Chapter 2.\\

Let $X,Y$ be complex Banach function spaces with absolute continuous norm over $\sigma$-finite measure spaces $(\Omega,\mu)$ and $(\widetilde{\Omega},\widetilde{\mu})$, respectively, and let $s\in [1,\infty]$. We note that the basic ideas and definitions presented in this chapter easily generalize to the more general setting of an abstract Banach lattice using the Krivine-calculus (cf. e.g. \cite{LT}, Section II.1.d).\\

\begin{definition}[$\mR_s$-boundedness]
Let $\mT\tm L(X,Y)$. The set $\mT$ is called {\em $\mR_s$-bounded}, if there exists a constant $C\in\R_{>0}$, such that for all $n\in\N,
T\in\mT^n$ and $x\in X^n$:
\begin{eqnarray} \label{def_Rs-bd_Summen}
&&\Big\| \Big(\sum_{j=1}^n |T_j x_j|^s \Big)^{1/s} \Big\|_Y \le C\, \Big\| \Big(\sum_{j=1}^n |x_j|^s \Big)^{1/s} \Big\|_X, \quad\text{ if } s<\infty,\\ \label{def_Rs-bd_Sup}
&&\Big\| \sup_{j\in\N_{\le n}} |T_j x_j| \Big\|_Y \le C\, \Big\| \sup_{j\in\N_{\le n}} |x_j|  \Big\|_X, \quad\text{ if } s=\infty.
\end{eqnarray}

The infimum of all such bounds $C$ is called the $\mR_s$-bound of $\mT$ and denoted by $\mR_s(\mT)$. If $T\in L(X,Y)$, we say that $T$ is $\mR_s$-bounded if $\{T\}$ is $\mR_s$-bounded and let $\mR_s(T):=\mR_s(\{T\})$.
\end{definition}

Obviously, if $\mT\tm \mL(X,Y)$ is $\mR_s$-bounded, then $\mT\tm L(X,Y)$, and $\mT$ is norm-bounded with $\sup_{T\in\mT} \|T\| \le \mR_s(\mT)$. Moreover, if $\mT\tm \mL(X,Y)$ and $C>0$, then $\mT$ is $\mR_s$-bounded with $\mR_s(\mT)\le C$ if and only if $\mT_0$ is $\mR_s$-bounded with $\mR_s(\mT_s)\le C$ for all finite subsets $\mT_0\tm\mT$, and in this case
\[
\mR_s(\mT) = \sup\{ \mR_s(\mT_0) \:|\: \mT_0\tm\mT \text{ finite} \}.
\]

If one considers $x\in X^n$ as an element of $M(\Omega,\C^n)$, we have
\[
\Big\| \Big(\sum_{j=1}^n |x_j|^s \Big)^{1/s} \Big\|_X = \|x\|_{X(\ell_n^s)}   \quad\text{and } \big\| \sup_{j\in\N_{\le n}} |x_j|^s \big\|_X = \|x\|_{X(\ell^\infty_n)},
\]

respectively. So $T\in\mT^n$ can be identified with the diagonal operator
\[
\widetilde{T}:X^n\to Y^n, x\mapsto (T_jx_j)_{j\in\N_{\le n}},
\]

which can be considered as a bounded operator $X(\ell^s_n)\to Y(\ell^s_n)$. With this notation the set  $\mT$ is $\mR_s$-bounded if and only if the set of operators
\[
\{\widetilde{T}\,|\, T\in\mT^n, n\in\N\} \tm \bigcup_{n\in\N} L\big( X(\ell^s_n),Y(\ell^s_n) \big)
\]

is uniformly bounded.

\begin{bem}
Since $X,Y$ have the Fatou property, we can replace the finite sums in (\ref{def_Rs-bd_Summen}) in the definition of $\mR_s$-boundedness by infinite series and the suprema in (\ref{def_Rs-bd_Sup}) by suprema over all $\N$. In particular, a single operator $T\in \mL(X,Y)$ is $\mR_s$-bounded if and only if the diagonal operator
\[
(x_n)_{n\in\N} \mapsto (Tx_n)_{n\in\N}
\]

induces a bounded operator $\widetilde{T}_s \in L( X(\ell^s),Y(\ell^s))$, and in this case $\mR_s(T) = \|\widetilde{T}_s\|_{L( X(\ell^s),Y(\ell^s))}$.
\end{bem}

An easy consequence is the following
\begin{prop} \label{RsL(X)-BS}
Let $\RsL(X,Y) := \{ T\in L(X,Y) \,|\, T$ is $ \mR_s$-$bounded\}$. Then $\RsL(X,Y)$, endowed with the norm $\mR_s(\mal)$, is a Banach space.
\qed\end{prop}

We note that in general $\RsL(X,Y)\neq L(X,Y)$ if $s\neq 2$, cf. Example \ref{bsp_operator_nicht_Rs-bdd} below. If the spaces $X,Y$ have appropriate geometric properties, norm-boundedness of a set of operators also implies the $\mR_s$-boundedness for a certain range of $s$:
\begin{bem}\label{bem_Rp_in_Lp}
Let $\mT\tm L(X,Y)$. Let $X$ be $p$-concave and $Y$ be $q$-convex for some $1\le p \le q \le \infty$, and let $\mT$ be norm-bounded. Then $\mT$ is $\mR_s$-bounded for all $s\in [p,q]$. In particular, if $X=L^p, Y=L^q$ and $\mT$ is norm-bounded, then $\mT$ is $\mR_s$-bounded for all $s\in [p,q]$.
\end{bem}

\begin{proof}
Let $s\in [p,q]$ and define $C:=\sup_{T\in\mT} \|T\|<\infty$. Let $n\in\N, T\in\mT^n$ and $x\in X^n$, then
\begin{eqnarray*}
\|\widetilde{T}x\|_{Y(\ell^s_n)} &\le& M^{(s)}(Y) \, \|\widetilde{T}x\|_{\ell^s_n(Y)} = M^{(s)}(Y) \, \big\|(\|T_jx_j\|_Y)_j\big\|_{\ell^s_n} \le  M^{(s)}(Y) \, \big\|( C \,\|x_j\|_X)_j\big\|_{\ell^s_n} \\
&=& M^{(s)}(Y) \,C\, \|x\|_{\ell^s_n(X)} \le C\,M^{(s)}(Y) M_{(s)}(X)  \mal \| x\|_{X(\ell^s_n)}.
\end{eqnarray*}

Here we used the fact that $X$ is also $s$-concave and $Y$ is $s$-convex since $p\le s\le q$, and $M_{(s)}(X), M^{(s)}(Y)<\infty$ denote the corresponding $s$-concavity/$s$-convexity constants of $X$ and $Y$, respectively. The supplementary assertions follow from the fact that $L^p$ is always $p$-concave and $p$-convex.
\end{proof}

We recall the related definition of $\mR$-boundedness (cf. e.g. \cite{clement} or \cite{levico}, Section 2): A set $\mT\tm L(X,Y)$ is called {\em $\mR$-bounded}, if there exists a constant $C\in\R_{>0}$, such that for all $n\in\N, T\in\mT^n$ and $x\in X^n$:
\begin{eqnarray}
&&  \E \Big| \sum_{j=1}^n r_j \tensor T_jx_j \Big|_X  \le C\,   \E \Big| \sum_{j=1}^n r_j \tensor x_j \Big|_X ,
\end{eqnarray}

where $(r_j)_{j\in\N}$ is any sequence of independent symmetric $\pm1$-valued, i.e. Bernoulli-distributed, random variables on some probability space, and $\E$ denotes the expectation with respect to the corresponding probability measure. We usually choose the Rademacher functions $r_j(t):=\sgn\sin(2^j\pi t), j\in\N$ on $[0,1]$.\\

Since the Khintchine-inequality holds in $r$-concave Banach function spaces if $r<\infty$, we obtain the following close relation between $\mR$-boundedness and $\mR_2$-boundedness:
\begin{remark} \label{bem_R-bd=R_2-bd}
Let $\mT\tm L(X,Y)$.
\begin{itemize}
\item[(1)] If $X$ is $r$-concave for some $r<\infty$ and $\mT$ is $\mR$-bounded, then $\mT$ is $\mR_2$-bounded,
\item[(2)] If $Y$ is $r$-concave for some $r<\infty$ and $\mT$ is $\mR_2$-bounded, then $\mT$ is $\mR$-bounded.
\end{itemize}

In particular, if both $X$ and $Y$ are $r$-concave for some $r<\infty$, then $\mT$ is $\mR$-bounded if and only if $\mT$ is $\mR_2$-bounded.
\end{remark}

We will now turn to some persistence properties of $\mR_s$-boundedness that correspond to persistence properties of $\mR$-boundedness, cf. e.g. \cite{levico}, Section 2. The first one is an immediate consequence of the norm-properties in the spaces $X(\ell^s_n)$:
\begin{prop}
Let $\mS \tm L(X,Y)$ be $\mR_s$-bounded.
\begin{itemize}
\item[(1)] If $\mT \tm L(X,Y)$ is $\mR_s$-bounded, then the set $\mS+\mT:= \{S+T \,|\, S\in\mS, T\in\mT\}$ is $\mR_s$-bounded,  and $\mR_s(\mS+\mT)\le \mR_s(\mS) + \mR_s(\mT)$.
\item[(2)] If $V$ is another Banach function space and $\mT \tm L(V,X)$ is $\mR_s$-bounded, then the set $\mS\mT :=  \{ST \,|\, S\in\mS, T\in\mT\}$ is $\mR_s$-bounded with $\mR_s(\mS\mT)\le \mR_s(\mS) \mal \mR_s(\mT)$.
\end{itemize}
\qed\end{prop}

The following convexity property is an important tool and is again a variant of a corresponding result for $\mR$-bounded operators, cf. e.g. \cite{levico}, Theorem 2.13.
\begin{prop} \label{Satz-aco-Rs-bd}
Let $\mT\tm L(X,Y)$ be $\mR_s$-bounded. Then the strong closure of the absolute convex hull $\overline{\aco}^s(\mT)\tm L(X,Y)$ is $\mR_s$-bounded with $\mR_s(\overline{\aco}^s(\mT))= \mR_s(\mT)$.
\qed\end{prop}

Let $(J,\nu)$ be a $\sigma$-finite measure space and $S:J \to L(X,Y)$ strongly measurable and $a\in L^1(J)$, then define an operator $T_{a,S}\in L(X,Y)$ by
\[
T_{a,S}x := \int_J a(t) S(t) x\, d\nu(t) \quad\text{ for all } x\in X.
\]

By approximation, it follows that $T_{a,S}\in \overline{\aco}^s (S(J))$ if $\|a\|_{L^1}=1$, which leads to the following
\begin{cor} \label{Rs-Bd-von-Operatoren-T_a,S}
Let $\mT\tm L(X,Y)$ be $\mR_s$-bounded, then the set
\[
\mS := \{ T_{a,S} \,|\, S:J \to L(X,Y) \text{ strongly measurable with }  S(J)\tm \mT , a\in L_1(J) \text{ with } \|a\|_{L_1}\le 1  \}
\]

is $\mR_s$-bounded with $\mR_s(\mS)\le \mR_s(\mT)$.
\end{cor}

This is proven in \cite{levico}, Corollary 2.14 for $\mR$-boundedness, and using Proposition \ref{Satz-aco-Rs-bd} the proof carries over to our situation.\\

We recall some standard notions about representation of functions with values in function spaces. If $J\tm\R$ is an interval and $x:J\to X$ is $\la$-measurable, then by the continuous embedding $X\into M(\mu)$ we can find a $\la\tensor\mu$-measurable representative $\tilde{x}:J\times\Omega\to\C$ and hence identify $x$ also with the measurable function $\omega \mapsto \tilde{x}(\mal,\omega)$. In this manner we can e.g. deal with the question if $x\in X(L^s_*(J))$. If there is no risk of confusion, we will work with this identification in the sequel without explicitly mentioning it.\\

By standard approximation arguments we obtain the following continuous version of $\mR_s$-boundedness.

\begin{prop} \label{prop_Rs-stetige-Fassung}
Let $s\in [1,\infty)$. Let $J\tm\R$ be a non-trivial interval and $S:J\to L(X,Y)$ be strongly measurable such that $S(J)$ is $\mR_s$-bounded. Then for all measurable $x:J\to X$ we have
\begin{equation} \label{Rs-stetige-Fassung}
\left\|\Big( \int_J |S(t)x(t)|^s \,dt\Big)^{1/s} \right\|_Y \le \mR_s(S(J)) \mal \left\|\Big( \int_J |x(t)|^s \,dt\Big)^{1/s} \right\|_X.
\end{equation}

In other words, the operator $S$ extends to a continuous diagonal operator $(S(t))_{t\in J}$ from $X(L^s(J))$ to $Y(L^s(J))$.
\end{prop}

The idea of proof is simple: We first consider suitable step functions, then the integrals turn to finite sums and we can use the definition of $\mR_s$-boundedness. Then we do an approximation argument to obtain the general result. For $s=2$, Proposition \ref{prop_Rs-stetige-Fassung} is proved in this way under slightly more restricting assumptions in \cite{Weis01}, Lemma 4~a, hence we suppress the technical details here.\\

We obtain an analogous result for $s=\infty$, in this case we can of course drop the measurability assumptions on $S$, and the proof is just an application of the Fatou property in $X$:
\begin{prop} \label{prop_Rs-stetige-Fassung,s=oo}
Let $J$ be a non-empty set and $S:J\to L(X,Y)$ such that $S(J)$ is $\mR_\infty$-bounded. Then for all mappings $x:J\to X$ we have
\begin{equation} \label{Rs-stetige-Fassung,s=oo}
\big\|  \sup_{t\in J} |S(t)x(t)|  \big\|_Y \le \mR_\infty(S(J)) \mal \big\|  \sup_{t\in J}|x(t)| \big\|_X.
\end{equation}
\qed\end{prop}

Further standard methods to obtain $\mR_s$-boundedness are by means of interpolation and duality. Recall that a set $\mT$ of operators is $\mR_s$-bounded if and only if the diagonal operators $\widetilde{T}$ for $T\in \mT^n, n\in\N$ induce uniformly bounded operators  from $X(\ell^s_n)$ to $Y(\ell^s_n)$. By complex interpolation we obtain $[X(\ell^{s_0}_n),X(\ell^{s_1}_n)]_\theta = X(\ell^{s_\theta}_n)$ (with equal norms) where $\frac{1}{s_\theta} = (1-\theta) \frac{1}{s_0} + \theta \frac{1}{s_1}$, cf. \cite{calderon}. This leads immediately to the following
\begin{prop}
Let $1\le s_0<s_1 \le \infty$. If $\mT\tm L(X,Y)$ is $\mR_{s_j}$-bounded for $j=1,2$, then $\mT$ is $\mR_s$-bounded for all $s\in [s_0,s_1]$.
\qed\end{prop}

In the special case $X=Y=L^p$, a norm-bounded set $\mT\tm X$ is always $\mR_p$-bounded, as we have seen in Remark \ref{bem_Rp_in_Lp}. Since there are various results for $\mR$-boundedness of operators, sometimes the following remark is helpful.
\begin{cor} \label{cor_R-bd->Rs-bd}
Let $X=L^p$ and $\mT\tm L(X)$ be $\mR$-bounded. Then $\mT$ is $\mR_s$-bounded for all $s\in [2\wedge p,2\vee p]$.
\qed\end{cor}

In particular, finite operator sets in $L^p$ are always $\mR_s$-bounded for all $s\in [2\wedge p,2\vee p]$, since the latter is true for $\mR$-boundedness. However it is well known that in spite of these special cases, finitely many or even a single operator need not to be $\mR_s$-bounded, even in the Hilbert space case.\\

\begin{example} \label{bsp_operator_nicht_Rs-bdd}
Consider $X=Y=L^2([0,1])$. Let $r_j(t) := \sgn(\sin(2^j\pi t))$ for all $t\in [0,1], j\in\N$ be the Rademacher functions. Then $(r_j)_{j\in\N}$ is an orthonormal system in $X$. Let us further define $I_j:= (2^{-j},2^{-j+1}]$, $f_j(t):=2^{j/2}\1I_{I_j}$ for all $j\in\N$, then it is easily checked that $(f_j)_{j\in\N}$ is an orthonormal system in $X$ as well. For all $n\in\N$, we have
\begin{eqnarray*}
\Big\| \Big( \sum_{j=1}^n |r_j|^s \Big)^{1/s}\Big\|_{L_2} &=& \Big\| \Big( \sum_{j=1}^n \1I_{[0,1]} \Big)^{1/s}\Big\|_{L_2} = \big\|
n^{1/s} \1I_{[0,1]} \big\|_{L_2} = n^{1/s}.
\end{eqnarray*}

Because the $f_j$ have disjoint supports, we have on the other hand
\begin{eqnarray*}
\Big\| \Big( \sum_{j=1}^n |f_j|^s \Big)^{1/s}\Big\|_{L_2} &=& \bigg( \int_0^1 \Big( \sum_{j=1}^n |f_j|^s \Big)^{2/s} \bigg)^{1/2} = \bigg(  \sum_{j=1}^n \int_0^1 \big( |f_j|^s \big)^{2/s} \bigg)^{1/2} = n^{1/2}.
\end{eqnarray*}

As the $(r_j)_j, (f_j)_j$ are orthonormal sequences, we can construct operators $T,S$ on $X$ with the property $Tf_j= r_j$ and $Sr_j=f_j$ for all $j\in\N$. Then the above equalities show that $T$ is not $\mR_s$-bounded in case $s<2$ and $S$ is not $\mR_s$-bounded in case $s>2$.
\end{example}

We now have a look at duality. Recall that the dual space $X'$ can be identified with the so-called associated Banach function space $X^*$ of $X$ since $X$ has absolute continuous norm, cf. \cite{bennett}, Chapter 1, and for $s\in [1,\infty)$ we have in this sense $(X(\ell^s))' = X'(\ell^{s'})$ by \cite{gretsky-uhl} since $\ell^s$ has the Radon-Nikodym property if $s\in [1,\infty)$. Moreover, for $T\in L(X,Y)$ we identify its dual operator $T'$ with the corresponding operator $T^*:Y^*\to X^*$. Then we obtain the following duality result.

\begin{satz} \label{satz_Rs-dual}
Let $s\in [1,\infty)$ and $\mT\tm L(X)$ be $\mR_s$-bounded. Then $\mT' := \{T' \,|\, T\in\mT\}$ is $\mR_{s'}$-bounded in $X'$.
\qed\end{satz}

We note that in general the associated Banach function space $X^*$ does not have absolute continuous norm and hence does not fit in our framework. So if we use duality results like Proposition \ref{satz_Rs-dual}, we usually require $X$ to be reflexive, so $X^*$ has also absolute continuous norm, cf. \cite{gretsky-uhl}.\\

We quote another useful criterion to check $\mR_s$-boundedness for concrete operators. Recall that a linear operator $S:X\to Y$ is called {\em positive}, if $x\ge 0$ implies $Sx\ge 0$ for all $x\in X$, and positive operators are always bounded. Then we obtain the following criterion for $\mR_s$-boundedness:
\begin{prop} \label{satz_Rs-erhalten-unter-Domination}
Let $\mT\tm L(X,Y)$ and $S:X\to Y$ be a positive operator that dominates $\mT$, i.e. $|Tx|\le S|x|$ for all $T\in\mT, x\in X$. Then $\mT$ is $\mR_s$-bounded with $\mR_s(\mT)\le \|S\|$ for all $s\in [1,\infty]$.
\end{prop}

The proof is based on the general result that positive operators $S:X\to Y$ always have bounded extensions $S\tensor \Id_E$ in the vector-valued spaces $X(E) \to Y(E)$ (cf. e.g. \cite{garcia}, Theorem V.1.12) and the obvious, but useful fact, that $\mR_s$-boundedness is inherited by domination in the following sense.
\begin{remark}\label{bem_Rs_domin}
Let $\mT,\mS\tm L(X,Y)$ such that for all $T\in \mT$ there is an $S\in\mS$ such that $|Tx|\le |Sx|$ for all $x\in X$. Then $\mT$ is
$\mR_s$-bounded if $\mS$ is $\mR_s$-bounded.
\qed\end{remark}

We conclude this section with some classical examples of $\mR_s$-bounded operators in the case $X=Y=L^p(\Omega)$ based on classical harmonic analysis. More involved examples can be found in \cite{artikel-Rs-Hoo}.\\

Let $(\Omega,d)$ be a metric space and $\mu$ be a $\sigma$-finite regular Borel measure on $\Omega$ such that $(\Omega,d,\mu)$ is a space of homogeneous type in the sense of Coifman and Weiss, cf. \cite{coifmanweiss}, \cite{coifmanweiss2}, i.e. there is a constant $D>0$, which will be called the dimension of $X$, and a constant $C_D\ge 1$ such that
\begin{equation} \label{doubling_prop2}
\mu(B(x,\la r)) \le C_D\,\la^D\, \mu(B(x,r))\quad\text{ for all } x\in\Omega, r>0, \la\ge 1.
\end{equation}

The (uncentered) Hardy-Littlewood maximal operator is given by
\[
(\mM f)(x) := \sup\left\{ \frac{1}{\mu(B)} \int_B |f|\,d\mu \:\Big|\: B\tm \Omega \text{ is a ball with } x\in B  \right\} \quad\text{for all } f\in L^1_{loc}(\Omega), x\in\Omega.
\]

Then the following classical theorem holds.
\begin{theorem}[Fefferman-Stein] \label{theorem_Fefferman-Stein}
Let $p\in (1,\infty)$ and $s\in (1,\infty]$, then the sublinear operator $\mM$ is $\mR_s$-bounded in $L^p(\Omega)$, i.e. there is a constant $C>0$ such that
$$
\forall\, n\in\N\, \forall\, f\in \big(L^p(\Omega)\big)^n\,:\; \Big\| \Big(\sum_{j=1}^n |\mM f_j|^s\Big)^{1/s} \Big\|_p \le C \mal  \Big\| \Big(\sum_{j=1}^n |f_j|^s\Big)^{1/s} \Big\|_p.
$$
\end{theorem}

This originates in the paper \cite{fefferman-stein} for the case $\Omega=\R^d$, and the generalization on spaces of homogeneous type can be found in \cite{GrafakosLiuYang} , Theorem 1.2. The $\mR_s$-boundedness of the Hardy-Littlewood maximal operator together with Proposition \ref{satz_Rs-erhalten-unter-Domination} above yields a wide class of examples for $\mR_s$-bounded sets of classical operators by well known uniform estimates against the maximal function.\\

Another large class of $\mR_s$-bounded operators is given by (singular) integral operators $T$ which are bounded in $L^r(\Omega)$ for some $r\in (0\infty)$ and have a kernel that fulfills the H\"ormander condition, e.g. \emph{Calder\'on-Zygmund operators}. Then it is well known that $T$ has bounded vector-valued extensions and hence is in particular $\mR_s$-bounded in $L^p(\Omega)$ for all $p,s\in (1,\infty)$, cf. \cite{GrafakosLiuYang} for this result and e.g.  \cite{stein-dick}, \cite{garcia}, \cite{grafakos} or \cite{duoandikoetxea} for the classical theory of Calder\'on-Zygmund operators.\\

Finally we note that in the monograph \cite{garcia} this topic (without using the terminology of $\mR_s$-boudnedness) is intensively studied in connection with weighted estimates, also involving Muckenhoupt weights, and in the article \cite{garcia-artikel} the topics from the monograph \cite{garcia} are also considered in the more general framework of Banach function spaces.

\section{$\mR_s$-sectorial operators} \label{section_Rs-Sekt.}

Let $X$ be a Banach function space with absolute continuous norm. In this section, $A:X\om D(A)\to X$ will always denote an injective  sectorial operator with type $\omega(A)$ and with dense domain and range.

\subsection{Definition and elementary properties of $\mR_s$-sectorial operators}

\begin{definition}
Let $s\in [1,\infty]$. The operator $A$ is called \emph{$\mR_s$-sectorial}, if there exists an $\omega \in [0,\pi)$ such that $\sigma(A)\tm \overline{\Sigma}_{\omega}$ and the set $\{zR(z,A) \,|\, z \in\C\ohne\overline{\Sigma}_{\sigma}\}$ is $\mR_s$-bounded for each $\sigma \in (\omega,\pi)$. The infimum $\omega_{\mR_s}(A)$ of all such $\omega$ is called the $\mR_s$-type of $A$. \\

In this case we define
\[
M_{s,\sigma}(A):= M_{\mR_s,\sigma}(A):= \mR_s( \{ zR(z,A), AR(z,A) \;|\; z \in \C\ohne \overline{\Sigma}_{\sigma} \})
\]

for all $\sigma \in (\omega_{\mR_s}(A),\pi)$. Observe that this set is indeed also $\mR_s$-bounded, since
\[
AR(z,A) = z R(z,A) -\id_X \quad\text{for all }  z \in \C\ohne \overline{\Sigma}_{\sigma},
\]

hence  $M_{\mR_s,\sigma}(A)\le \mR_s( \{ zR(z,A) \;|\; z \in \C\ohne \overline{\Sigma}_{\sigma} \}) +1 \le 2M_{\mR_s,\sigma}(A)$.
\end{definition}

We will now introduce a diagonal operator extension $\widetilde{A}_s$ of $A\tensor\Id_{\ell^s}$ such that properties as $\mR_s$-sectoriality and $\mR_s$-boundedness of the \HU-calculus can be expressed as "simple" sectoriality and boundedness of the \HU-calculus for the single operator $\widetilde{A}_s$. Although this is a natural concept, we note that it is not always straightforward to check these correspondences. Indeed, while $\mR_s$-sectoriality of $A$ will turn out to imply sectoriality of $\widetilde{A}_s$ by definition, the converse is not clear, since $\mR_s$-boundedness of the \emph{single operators} $\la R(\la,A)$ might not imply $\mR_s$-boundedness of \emph{the set} $\{ \la R(\la,A) \:|\: \la \in \C\ohne \overline{\Sigma}_\sigma\}$.

\begin{defsatz} \label{defsatz_A-Schlange-in-X(ls)}
Let $s\in [1,\infty]$ and assume that $A$ is $\mR_s$-sectorial. Define the diagonal operator
\begin{equation}
\widetilde{A}_s := \{ ((x_n)_{n\in\N}, (A x_n)_{n\in\N}) \,|\,  (x_n)_{n\in\N} \in X(\ell^s), x_n\in D(A) \text{ for all } n\in\N \text{ and } (Ax_n)_{n\in\N} \in X(\ell^s) \}
\end{equation}

in $X(\ell^s)$. Then $\widetilde{A}_s$  is a sectorial operator in $X(\ell^s)$ with $\omega(\widetilde{A}_s)\le \omega_{\mR_s}(A)$, and  we have
\begin{equation} \label{gl_resolventen-von-A-Schlange_s}
\forall\, \la\in\C\ohne \overline{\Sigma}_{\omega_{\mR_s}(A)}, x\in X(\ell^s) \,:\:  R(\la,\widetilde{A}_s)x = (R(\la,A)x_n)_{n\in\N}.
\end{equation}
\end{defsatz}

\begin{proof}
Let $\sigma\in(\omega_{\mR_s}(A),\pi)$. For all $\la\in \C\ohne \overline{\Sigma}_\sigma=:S$ define the operator $R(\la) := \la R(\la,A)$, then $R(S)$ is $\mR_s$-bounded, so let $M:= \mR_s(R(S))$. Moreover, for each $\la\in S$ and $x=(x_n)_{n\in\N} \in X(\ell^s)$ define
\begin{equation*}
\widetilde{R}(\la)x := (R(\la)x_n)_{n\in\N}.
\end{equation*}

Then $\widetilde{R}(\la) \in L(X(\ell^s))$, and the operator set $\widetilde{R}(S)$ is uniformly bounded by $M$. Moreover it is an easy calculation that $\widetilde{R}(\la)= \la R(\la,\widetilde{A}_s)$, hence $\widetilde{A}_s$ is sectorial with $\omega(\widetilde{A}_s) \le \sigma$.
\end{proof}

As already pointed out, the reverse conclusion might be false; one way of overcoming this problem would be a modification of the vector-valued operator $\widetilde{A}$, e.g. by defining $\widetilde{A}x := (w_jx_j)_j$, where $\{w_j\,|\, j\in\N\}$ is a dense subset of a sector $\Sigma_\omega$, cf. \cite{arendt-bu-tools-max-reg} for such an approach in a different context.\\

The representation (\ref{gl_resolventen-von-A-Schlange_s}) of the resolvents of $\widetilde{A}_s$ implies that also the functional calculus for $\widetilde{A}_s$ is just given by diagonal operators with maximal domains. Recall that $\BB(\Sigma_\sigma)$ is the algebra of analytic functions on the sector $\Sigma_\sigma$ that are polynomially bounded at $0$ and $\infty$, cf. Subsection \ref{subsection-Hoo}.

\begin{lemma} \label{lemma_funktionalkalkuel-fuer-A-Schlange_s}
Let $s\in [1,\infty]$. Assume that $A$ is $\mR_s$-sectorial and let $\sigma\in(\omega_{\mR_s}(A),\pi]$ and $f\in \BB(\Sigma_\sigma)$. Then
\[
D(f(\widetilde{A}_s)) = \{  (x_n)_{n\in\N} \in X(\ell^s) \:|\:  x_n\in D(f(A)) \text{ for all } n\in\N \text{ and } (f(A)x_n)_{n\in\N} \in X(\ell^s)  \}
\]

and $\ds f(\widetilde{A}_s)x = (f(A)x_n)_{n\in\N} $ for all $x=(x_n)_{n\in\N} \in D(f(\widetilde{A}_s))$.
\end{lemma}

\begin{proof}
Let first $\ph\in H_0^\infty(\Sigma_\sigma)$. Let $\omega\in (\omega_{\mR_s}(A),\sigma)$ and $\Gamma$ be the usual parametrization of the boundary $\del\Sigma_\omega$. Since the projections $\pi_k: X(\ell^s)\to X, (x_n)_{n\in\N}\mapsto x_k$ are continuous for all $k\in \N$, the representation (\ref{gl_resolventen-von-A-Schlange_s}) of the resolvents of $\widetilde{A}_s$ implies
\[
\ph(\widetilde{A}_s)x = \frac{1}{2\pi i} \int_\Gamma \ph(z) R(z,\widetilde{A}_s)x\, dz = \bigg( \frac{1}{2\pi i} \int_\Gamma \ph(z) R(z,A)x_n\, dz \bigg)_{n\in\N} = (\ph(A)x_n)_{n\in\N}
\]

for all $x=(x_n)_{n\in\N} \in X(\ell^s)$. Now consider the general case $f\in \BB(\Sigma_\sigma)$. Let $\rho(z) := z(1+z)^{-2}$ and choose $m\in\N_0$ such that $\ph:= \rho^m f\in H_0^\infty(X(\ell^s))$. Then
\[
f(\widetilde{A}_s) = (\rho(\widetilde{A}_s))^{-m} (\rho^m f)(\widetilde{A}_s) = \big(\widetilde{A}_s(1+\widetilde{A}_s)^{-2}\big)^{-m} \, \ph(\widetilde{A}_s) = \big((1+\widetilde{A}_s)^{2}\widetilde{A}_s^{-1}\big)^{m} \, \ph(\widetilde{A}_s).
\]

This yields the claim since $\big((1+\widetilde{A}_s)^{2}\widetilde{A}_s^{-1}\big)^{m}$ is a diagonal operator with maximal domain.
\end{proof}

We will now turn to some elementary properties of $\mR_s$-sectorial operators that are standard for sectorial operators or e.g. $\mR$-sectorial operators. In fact, many properties can be proven analogously as it is done for $\mR$-sectorial operators in \cite{levico}, Chapter 2, hence we will often refer to the proofs given there.\\

The same arguments as for $\mR$-sectorial operators show the following
\begin{remark}
If the set $\{ t(t+A) \,|\, t>0 \}$ is $\mR_s$-bounded, then $A$ is $\mR_s$-sectorial.
 \qed\end{remark}

Using the elementary functional calculus for $H_0^\infty$-functions we can extend $\mR_s$-boundedness to more general sets consisting of operators generated by functions of $A$.
\begin{lemma} \label{Lemma_ph(zA)-Rs-bd}
Let $s\in [1,\infty]$ and $A$ be an $\mR_s$-sectorial operator in $X$. Let $\sigma>\omega_{\mR_s}(A)$ and $\mF\tm H_0^\infty(\Sigma_\sigma)$ such that one can choose $C_0,\beta>0$ with $|\ph(z)|\le C_0 \, (|z|^\beta \wedge |z|^{-\beta})$ for all $z\in\Sigma_\sigma, \ph\in\mF$. Then for each $0\le \delta< \sigma-\omega_{\mR_s}$, the set
\[
\{\ph(zA) \,|\, z\in\Sigma_\delta, \ph\in\mF\}
\]

is $\mR_s$-bounded. To be more precise, for each $\omega\in (\omega_{\mR_s}(A),\sigma)$ we can choose a constant $C=C(\omega,C_0,\beta)$ such that the estimate
\[
\mR_s(\{\ph(zA) \,|\, z\in\Sigma_\delta, \ph\in\mF\}) \le C\, M_{s,\sigma}(A)
\]

holds for all $\delta\in [0,\sigma-\omega)$.
\end{lemma}

\begin{proof}
Let $\sigma-\delta > \omega > \omega_M(A)$. Let $z\in \Sigma_\delta$, then $z=\tau w$ for some $\tau>0, w\in \Sigma_\delta\cap S^1$, and $w\la\in\Sigma_\sigma$ for all $\la\in\del\Sigma_{\omega}$, hence
\begin{eqnarray*}
\int_{\del\Sigma_\omega} |\ph(z\la)| \, \frac{|d\la|}{|\la|} &=& \sum_{j\in \{-1,1\}} \int_0^\infty |\ph(\tau twe^{ij\omega})| \, \frac{dt}{t} = \sum_{j\in \{-1,1\}} \int_0^\infty |\ph(twe^{ij\omega})| \, \frac{dt}{t}\\
&\le& C_0 \, \sum_{j\in \{-1,1\}} \int_0^\infty (t^\beta \wedge t^{-\beta}) \, \frac{dt}{t} =: M < \infty,
\end{eqnarray*}

i.e. $\|\ph(z \mal)\|_{L^1\big(\del\Sigma_\omega,|\frac{d\la}{\la}| \big))} \le M$ for all $z\in\Sigma_\delta, \ph\in\mF$. We have
\[
\ph(zA) = \frac{1}{2\pi i} \int_{\del\Sigma_\omega} \ph(z\la) R(\la,A) \, d\la =  \frac{1}{2\pi i} \int_{\del\Sigma_\omega} \ph(z\la) \mal \la R(\la,A) \, \frac{d\la}{\la},
\]

hence $|\ph(zA)x| \le \frac{1}{2\pi} \, \int_{\del\Sigma_\omega} |\ph(z\la)| \mal |\la R(\la,A)x| \, \abs{\frac{d\la}{\la}}$ for each $\ph\in \mF, z\in\Sigma_\delta$ and $x\in X$. Since $\{\la R(\la,A) \,|\, \la\in\del\Sigma_\omega\}$ is $\mR_s$-bounded, the assertion follows with Corollary \ref{Rs-Bd-von-Operatoren-T_a,S}.
\end{proof}

In view of the functional calculus for the extended Dunford-Riesz class $\mE(\Sigma_\sigma)$ (cf. Subsection \ref{subsection-Hoo}) we obtain the following slightly more general version.
\begin{cor} \label{cor-glm-mE}
Let $s\in [1,\infty]$ and $A$ be an $\mR_s$-sectorial operator in $X$. Let $\sigma>\omega_{\mR_s}(A)$ and $\mF\tm \mE(\Sigma_\sigma)$ such that there exists an $\eps>0$ with
\begin{equation} \label{cor-glm-mE-eq}
M_0 := \sup_{f\in\mF} \|f\|_{\infty,\sigma} < \infty \quad\text{ and } C_0:=\sup_{f\in\mF}\sup_{z\in\Sigma_\sigma} ( |z|^\eps \vee |z|^{-\eps}) \Big| f(z) - \frac{f(0) + f(\infty)z}{1+z} \Big| < \infty.
\end{equation}

Let $0\le \delta< \sigma-\omega_{\mR_s}$, then the set $\{ f(zA) \,|\, z\in\Sigma_\delta, f\in\mF \}$ is $\mR_s$-bounded.
\end{cor}

\begin{proof}
Let $f\in\mF$, then by Subsection \ref{subsection-Hoo} we have a decomposition $f(\zeta) = \ph_f(\zeta) + \frac{a}{1+\zeta} + b$, where $\ph\in H_0^\infty(\Sigma_\sigma)$ and $b= f(\infty), a+b = f(0)$. Hence $|a|,|b|\le 2M_0$ and
\[
|\ph_f(\zeta)| = \Big| f(\zeta) -  \frac{(a+b)+b\zeta}{1+\zeta}  \Big| \le C_0\, (|z|^\eps \wedge |z|^{-\eps}).
\]

For each $z\in \Sigma_\delta$ we obtain the representation $f(z\mal)= \ph_f(z\mal) +  \frac{a}{1+z\mal} + b$, hence
\[
f(zA) = \ph_f(zA) +  a (1+zA)^{-1} + b\Id_{Y} = \ph_f(zA) +  a \la R(\la,A) + b\Id_{Y},
\]

where $\la := -\frac1z \in \C\ohne\overline{\Sigma_{\pi-\delta}}$. Since $\omega' := \pi-\delta \ge \sigma-\delta >\omega_M(A)$, the set
\[
\Big\{ a\la R(\la,A) \;\big|\; -\frac1\la \in \Sigma_\delta, |a|\le 2M_0 \Big\}
\]

is $\mR_s$-bounded, and Lemma \ref{Lemma_ph(zA)-Rs-bd}  implies the $\mR_s$-boundedness of $\{ \ph_f(zA) \,|\, z\in\Sigma_\delta, f\in\mF \}$, hence also $\{ f(zA) \,|\, z\in\Sigma_\delta, f\in\mF\}$ is $\mR_s$-bounded.
\end{proof}

A special case is the following corollary with just one function $f\in \mE(\Sigma_\sigma)$.
\begin{cor}
Let $s\in [1,\infty]$ and $A$ be an $\mR_s$-sectorial operator in $X$. Let $\sigma>\omega_{\mR_s}(A)$, $0\le \delta< \sigma-\omega_{\mR_s}$ and $f\in \mE(\Sigma_\sigma)$. Then the set $\{ f(zA) \:|\: z\in\Sigma_\delta\}$ is $\mR_s$-bounded.
\qed\end{cor}

An immediate consequence is the following: If $\omega_{\mR_s}(A)<\pi/2$, the generated analytic semigroup $(e^{-zA})_{z\in\Sigma_\delta}$ is $\mR_s$-bounded for all $\delta \in [0,\pi/2 - \omega_{\mR_s}(A))$. We will say in this case that $A$ is \emph{$\mR_s$-analytic}. More characterizations of $\mR_s$-analyticity  are the content of the following proposition, which is an $\mR_s$-bounded version of \cite{levico}, Theorem 2.20, and again the proof given there carries over to our situation if we replace "$\mR$-boundedness" by "$\mR_s$-boundedness" and take into account Corollary \ref {Rs-Bd-von-Operatoren-T_a,S}.
\begin{prop} \label{satz_Charakterisierung-Rs-analytisch}
Assume $\omega(A)<\pi/2$. Then the following conditions are equivalent:
\begin{itemize}
\item[(1)] $A$ is $\mR_s$-analytic,
\item[(2)] $A$ is $\mR_s$-sectorial with $\omega_{\mR_s}(A)<\pi/2$,
\item[(3)] The set $\{t^n (it+A)^{-n} \;|\; t\in\R\ohne\{0\}\}$ is $\mR_s$-bounded for some $n\in\N$,
\item[(4)] The sets $\{e^{-tA} \:|\: t>0\}, \{tAe^{-tA} \:|\: t>0\}$ are $\mR_s$-bounded.
\end{itemize}
\qed\end{prop}

For concrete examples of $\mR_s$-sectorial operators we refer to Subsection \ref{subsection_Rs-bd-Hoo} and the literature cited there.

\subsection{Equivalence of $s$-power function norms} \label{subsection_Equivalence-of-s-power-function-norms}

We will now turn to the central estimates for a reasonable definition of the associated $s$-intermediate spaces for an $\mR_s$-sectorial operator, which will be done in the next section. The following proposition is well known for $s=2$ and $X=L^p, T=\Id$, cf. \cite{LeM04}, Theorem 1.1.
\begin{satz} \label{satz-norm-equiv-und-Hoo}
Let $s\in [1,\infty]$  and $A$ be an $\mR_s$-sectorial operator in $X$. Let $\sigma>\omega_{\mR_s}(A)$ and $\ph,\psi\in H_0^\infty(\Sigma_\sigma)\ohne \{0\}$. Then there is a constant $C>0$ such that for all $f\in H^\infty(\Sigma_\sigma)$ and for any $\mR_s$-bounded operator $T\in L(X)$ and $x\in X$ we have
\begin{eqnarray} \label{norm-equiv-und-Hoo}
\bigg\| \Big( \int_0^\infty |f(A)\ph(tA)Tx|^s \,\frac{dt}{t} \Big)^{1/s}\bigg\|_X &\le&  C \,\mR_s(T)\, \|f\|_{\infty,\sigma} \, \bigg\| \Big( \int_0^\infty |\psi(tA)x|^s \,\frac{dt}{t} \Big)^{1/s}\bigg\|_X
\end{eqnarray}

(with the usual modification if $s=\infty$).
\end{satz}

\begin{remark}
The norm expressions occurring in the estimate (\ref{norm-equiv-und-Hoo}) will also be referred to as \emph{$s$-power function norms}.
\end{remark}

\begin{proof}[Proof of Proposition \ref{satz-norm-equiv-und-Hoo}]
We will shorten arguments and calculations in some places, since the proof of Proposition \ref{satz-norm-equiv-und-Hoo} works in the same way as the proof given in \cite{LeM04} for the case $X=L^p$ and $s=2$. In some cases one just has to replace Cauchy's inequality by H\"older's inequality.\\

By an approximation argument, one only has to consider the case $f\in H_0^\infty(\Sigma_\sigma)$. Let $x\in X$ such that $\|x\|_\psi := \| ( \int_0^\infty |\psi(tA)x|^s \,\frac{dt}{t} )^{1/s}\|_X < \infty$. Let $\omega \in (\omega_{\mR_s}(A),\sigma)$ and $\Gamma$ be the parameterized contour of $\del\Sigma_\omega$. Choose auxiliary functions $F,G\in H_0^\infty(\Sigma_\sigma)$ such that
\begin{equation} \label{gl_Zerlegung-der-Eins-Integralformel}
\int_0^\infty F(t)G(t)\psi(t)\,\frac{dt}{t} =1.
\end{equation}

Then for each $H\in \{F,G\}$ we have
\[
\sup_{t >0} \int_\Gamma |H(tz)|\, \frac{|dz|}{|z|} < \infty \quad\text{and } \sup_{z\in\del\Sigma_\omega} \int_0^\infty |H(tz)|\,\frac{dt}{t} < \infty,
\]

Analogously as in \cite{LeM04} we will proceed in four steps, where the constants $C_k$ wich occur in each step do not depend on $x$ and $f$.\\

{\em Step 1.} For all $t>0$ we have
\begin{equation}
S(t):=f(A)G(tA)T = \frac{1}{2\pi i} \int_\Gamma f(z) G(tz) zR(z,A)T \,\frac{dz}{z}.
\end{equation}

Then
\[
\| f(\mal) G(t\mal)\|_{L^1\big(\Gamma,  |\frac{dz}{z}| \big)} = \int_\Gamma |f(z) G(tz)| \,\frac{|dz|}{|z|} \le \Big( \sup_{r>0} \int_\Gamma |G(rz)|\, \frac{|dz|}{|z|} \Big) \mal  \|f\|_{\infty,\sigma}.
\]

Since $\mS:=\{zR(z,A)T \,|\, z\in \Gamma \}$ is $\mR_s$-bounded by assumption, Corollary \ref{Rs-Bd-von-Operatoren-T_a,S} yields that also $S((0,\infty))$ is $\mR_s$-bounded with $\mR_s(S((0,\infty))) \le C_1 \,\mR_s(T)\, \|f\|_{\infty,\sigma}$. Hence by Proposition \ref{prop_Rs-stetige-Fassung} we obtain
\begin{eqnarray} \nonumber
\Big\| \Big( \int_0^\infty |f(A)G(tA)\psi(tA)x|^s \,\frac{dt}{t} \Big)^{1/s}\Big\|_X &=& \Big\| \Big( \int_0^\infty |S(t)\psi(tA)x|^s \,\frac{dt}{t} \Big)^{1/s}\Big\|_X \\ \label{Step1-1}
&\le& C_2 \mal \mR_s(T)\mal \|f\|_{\infty,\sigma}\mal \|x\|_\psi.
\end{eqnarray}

(with the usual modification if $s=\infty$ using Proposition \ref{prop_Rs-stetige-Fassung,s=oo}.)\\

{\em Step 2.} Let $w(t):= S(t)\psi(tA)x$ for all $t>0$ and $u(z):= \ds\int_0^\infty F(tz)w(t)\,\frac{dt}{t}$. By choosing appropriate representatives, by H\"older's inequality and Fubini's theorem we have $\mu$-a.e. for $s<\infty$:
\begin{eqnarray*}
\int_\Gamma |u(z)|^s \, \frac{|dz|}{|z|} &\le& \int_\Gamma\, \Big( \int_0^\infty |F(tz)w(t)| \,\frac{dt}{t} \Big)^s \frac{|dz|}{|z|} \\
&\le& \sup_{z\in\Gamma} \Big(\int_0^\infty |F(tz)|\, \frac{dt}{t} \Big)^{s-1} \mal \Big( \sup_{t>0} \int_\Gamma |F(tz)|\,\frac{|dz|}{|z|} \Big)\mal    \int_0^\infty|w(t)|^s \, \frac{dt}{t},
\end{eqnarray*}

hence we obtain the estimate
\begin{equation} \label{Step2-1}
\Big\| \Big( \int_\Gamma |u(z)|^s \, \frac{|dz|}{|z|} \Big)^{1/s}\Big\|_X \le C_3 \mal \Big\| \Big( \int_0^\infty |w(t)|^s \, \frac{dt}{t}\Big)^{1/s}\Big\|_X \stackrel{(\ref{Step1-1})}{\le} C_4\mal \mR_s(T) \mal  \|f\|_{\infty,\sigma}\mal \|x\|_\psi.
\end{equation}

If $s=\infty$, we obtain similarly
\[
\sup_{z\in\Gamma} |u(z)| \le \Big( \sup_{z\in\Gamma} \int_0^\infty |F(tz)|\, \frac{dt}{t} \Big) \mal \sup_{t>0} |w(t)|,
\]

hence also
\begin{equation} \label{Step2-1-oo}
\big\| \sup_{z\in\Gamma} |u(z)| \big\|_X \le C_3\mal \big\|  \sup_{t>0} |w(t)| \big\|_X \stackrel{(\ref{Step1-1})}{\le} C_4 \mal \mR_s(T) \mal  \|f\|_{\infty,\sigma}\mal \|x\|_\psi.
\end{equation}

{\em Step 3.} Let $v(t):= \ds\int_\Gamma \ph(tz)zR(z,A)u(z)\,\frac{dz}{z}$ for all $t>0$. Then again, with H\"older's inequality and Fubini's theorem we can conclude if $s<\infty$:
\begin{eqnarray*}
\int_0^\infty |v(t)|^s\, \frac{dt}{t} &\le & \int_0^\infty \Big( \int_\Gamma |\ph(tz)|\, |zR(z,A)u(z)|\,\frac{|dz|}{|z|}\Big)^s    \,\frac{dt}{t}\\
&\le& \sup_{t>0}   \Big(\int_\Gamma |\ph(tz)| \,\frac{|dz|}{|z|}\Big)^{s-1} \mal \Big( \sup_{z\in\Gamma}  \int_0^\infty |\ph(tz)|\,\frac{dt}{t} \Big)  \mal \int_\Gamma |zR(z,A)u(z)|^s \,\frac{|dz|}{|z|}.
\end{eqnarray*}

Using again $\mR_s$-boundedness of $\{zR(z,A) \,|\, z\in \Gamma\}$ in the same way as in Step 1 we obtain
\begin{eqnarray} \nonumber
&&\Big\| \Big( \int_0^\infty |v(t)|^s\, \frac{dt}{t}\Big)^{1/s}\Big\|_X \le C_5 \mal \Big\| \Big(  \int_\Gamma |zR(z,A)u(z)|^s \,\frac{|dz|}{|z|} \Big)^{1/s}\Big\|_X \\ \label{gl-proof-equiv-step3}
&\le& C_5\mR_s(\mS) \mal  \Big\| \Big(  \int_\Gamma |u(z)|^s \,\frac{|dz|}{|z|} \Big)^{1/s}\Big\|_X
\stackrel{(\ref{Step2-1})}{\le} C_6\mal \mR_s(T)\mal  \|f\|_{\infty,\sigma}\mal \|x\|_\psi.
\end{eqnarray}

The analogous inequality holds also in the case $s=\infty$, which can be shown in the same manner as in Step 2.\\

{\em Step 4.} By analytic continuation we have
\[
\int_0^\infty F(tz)G(tz)\psi(tz)\,\frac{dt}{t} = 1 \quad\text{ for all } z\in\Sigma_\sigma.
\]

By the multiplicativity of the functional calculus (and Fubini) we obtain
\[
f(A) = \int_0^\infty F(tA)G(tA)\psi(tA)f(A)\,\frac{dt}{t},
\]

hence for all $\tau>0$:
\begin{eqnarray*}
f(A)\ph(\tau A)Tx &=& \int_0^\infty \ph(\tau A)F(tA)G(tA)\psi(tA)f(A)Tx \,\frac{dt}{t}\\
&=& \int_0^\infty \Big( \frac{1}{2\pi i} \int_\Gamma \ph(\tau z)F(tz)R(z,A)\,dz \Big) G(tA)\psi(tA)f(A)Tx \,\frac{dt}{t}\\
&=& \frac{1}{2\pi i} \int_\Gamma   \ph(\tau z)R(z,A)  \Big( \int_0^\infty F(tz) G(tA)\psi(tA)f(A)Tx \,\frac{dt}{t} \Big) \,dz \\
&=& \frac{1}{2\pi i} \int_\Gamma   \ph(\tau z)R(z,A) u(z) \,dz= v(\tau ).
\end{eqnarray*}

So with (\ref{gl-proof-equiv-step3}) the claim follows for $f\in H_0^\infty(\Sigma_\sigma)$.
\end{proof}

\subsection{Discrete $s$-power function norms} \label{subsection_Discrete_s-power_function_norms}

We now turn to a version of Proposition \ref{satz-norm-equiv-und-Hoo} dealing with discrete counterparts to the $s$-power function norms in estimate \ref{norm-equiv-und-Hoo} of the form $$\Big\| \Big( \sum_{j\in \Z} |\ph(2^j A)x|^s \Big)^{1/s} \Big\|_X.$$ Such a version cannot hold in full generality for arbitrary $\ph,\psi\in H_0^\infty(\Sigma_\sigma)\ohne\{0\}$, because such functions could have "too many zeros", and the crucial step (\ref{gl_Zerlegung-der-Eins-Integralformel}) in the proof of Proposition \ref{satz-norm-equiv-und-Hoo} would then break down. For a concrete counterexample consider $A=\Id_X$ and $\ph\in H_0^\infty(\Sigma_{\pi/4})$ with $\ph(2^j)=0$ for all $j\in\Z$. Hence, we will restrict ourselves to a suitable subclass of $H_0^\infty$-functions. The assumptions for this class of functions are rather technical and made to fit our needs, but they are not too restrictive: we will show that the standard $H^\infty_0$-functions we usually use as concrete auxiliary functions belong to this subclass of $H_0^\infty$.

\begin{definition} \label{def_UE}
Let $\sigma\in (0,\pi]$ and $\ph\in H_0^\infty(\Sigma_\sigma)$ with $0\notin\ph(\Sigma_\sigma)$. We say that $\ph$ belongs to the class $\Phi^\Sigma_{\sigma,0}$ if the following property (named after {\em uniform in $\mE$}) holds:
\begin{itemize}
\item[(UE)] There exist $d\in\Z$ and a set of functions $\mF\tm \mE(\Sigma_\sigma)$ such that $\mF$ fulfills condition (\ref{cor-glm-mE-eq}) from Corollary \ref{cor-glm-mE} and
\[
\big\{ \ph(2^{j}t\mal)/\ph(2^{j-d}\mal) ,  \ph(2^{j+d}\mal)/\ph(2^{j}t\mal)  \;\big|\; j\in\Z, t\in [1,2] \big\} \tm \{ f(r\mal) \:|\: f\in\mF, r>0\}.
\]
\end{itemize}
\end{definition}

The motivation of this definition is the following conclusion, which is an immediate consequence of Corollary \ref{cor-glm-mE}:
\begin{lemma} \label{lemma_operatoren-Sj(t)-sind-Rs}
Let $\sigma\in (0,\pi]$ and $\ph\in \Phi^\Sigma_{\sigma,0}$. Choose $d\in\Z$ due to property (UE) above and define the operators $S_j(t):= \big(\ph(2^{j+d}\mal)/\ph(2^{j}t\mal) \big)(A) $ and $\overline{S}_j(t) := \big( \ph(2^{j}t\mal)/\ph(2^{j-d}\mal)\big)(A)$ for all $j\in\Z, t\in[0,1]$. Then the set
\[
\{ S_j(t), \overline{S}_j(t) \,|\, j\in\Z, t\in [1,2]\}
\]

is $\mR_s$-bounded, and for all $ j\in\Z, t\in [1,2]$ we have
\begin{equation} \label{lemma_operatoren-Sj(t)-sind-Rs-eq}
\ph(2^{j}tA)\,S_j(t) = S_j(t) \ph(2^{j}tA) = \ph(2^{j+d}A), \quad\text{hence } S_j(t)  = \ph(2^{j}tA)^{-1}\ph(2^{j+d}A).
\end{equation}

and
\begin{equation} \label{lemma_operatoren-Sj(t)-sind-Rs-eq2}
\ph(2^{j-d}A)\,\overline{S}_j(t) = \overline{S}_j(t) \ph(2^{j-d}A) = \ph(2^{j}tA) \quad\text{hence } \overline{S}_j(t)  = \ph(2^{j-d}A)^{-1}\ph(2^{j}tA),
\end{equation}
\end{lemma}

In fact, instead of the technical condition (UE) we could assume for $\ph$ that the conclusion of Lemma \ref{lemma_operatoren-Sj(t)-sind-Rs} holds, since this is what we are interested in. Nevertheless, we want to formulate a condition that dependes only on the function $\ph$ and this is the reason for Definition \ref{def_UE}.\\

We now turn to the most important examples of functions in $\Phi^\Sigma_{\sigma,0}$.
\begin{bspe} \label{bspe_Standardfkt.-in-Phi^Sigma}
\begin{itemize}
\item[(1)] Let $\sigma\in (0,\pi)$ and $m\in\N$ and define $\ph(z):= \frac{z^m}{(1+z)^{2m}}$ for all $z\in\Sigma_\sigma$, then $\ph\in\Phi^{\Sigma}_{0,\sigma}$.
\item[(2)] Let $\sigma\in (0,\pi/2)$ and $\alpha>0$ and define $\ph(z):= z^\alpha e^{-z}$ for all $z\in\Sigma_\sigma$, then $\ph\in\Phi^{\Sigma}_{0,\sigma}$.
\end{itemize}

\begin{proof}
(1) We only consider the case $m=1$, where we will show that the condition (UE) is fulfilled with $d=0$. For $j\in\Z, t\in [1,2]$ and $z\in\Sigma_\sigma$ define
\[
g_{j,t}(z) := \ph(2^jtz)/\ph(2^jz) = t \, \Big(\frac{1+2^jz}{1+2^jtz}\Big)^2 =  t \, \Big(\frac{1+t^{-1}2^jtz}{1+2^jtz}\Big)^2.
\]

With $r=2^jt>0$ and $\tau := t^{-1} \in [1/2,1]$ we obtain $g_{j,t}(z) = t\mal f_\tau(rz)$ where $f_\tau(z):=\Big(\frac{1+\tau z}{1+z}\Big)^2$. We will show that $\mF :=\{t\mal f_{\tau} \,|\, t,\tau\in [1/2,2]\}$ fulfills condition (\ref{cor-glm-mE-eq}) from Corollary \ref{cor-glm-mE}. So let $\tau\in  [1/2,2]$, then $f_\tau(0)=1$ and $f_\tau(\infty)=\tau^2$, and
\begin{eqnarray*}
\psi_\tau(z) &:=& f_\tau(z) - \frac{f_\tau(0) + f_\tau(\infty)z }{1+z} = \frac{(1+\tau z)^2 - (1+z)(1+\tau^2 z)}{(1+z)^2} = (2\tau - \tau^2 - 1)\,\frac{z}{(1+z)^2},
\end{eqnarray*}

hence $|\psi_\tau(z)|\le |2\tau - \tau^2 - 1|\,\Big|\frac{z}{(1+z)^2}\Big| \le  9\,\Big|\frac{z}{(1+z)^2}\Big|$. This shows that the uniform estimate (\ref{cor-glm-mE-eq}) holds with $\eps=1$. Now consider
\[
h_{j,t}(z) := \ph(2^jz)/\ph(2^jtz) = t^{-1} \, \Big(\frac{1+2^jtz}{1+2^jz}\Big)^2,
\]

then $h_{j,t}(z) = \tau f_t(2^jz)$ with the same notations as above, hence also $\{h_{j,t} \,|\, j\in\Z, t\in [1,2]\}\tm \{f(r\mal) \,|\, f\in\mF \}$.\\

This shows that condition (UE) is fulfilled with $d=0$, hence $\ph\in \Phi^\Sigma_{\sigma,0}$ for $m=1$. The general case $m\in\N$ can be treated analogously, hence we omit the proof.\\

(2) For $j\in\Z, t\in [1,2]$ and $z\in\Sigma_\sigma$ define
\[
g_{j,t}(z) := \ph(2^{j}tz)/\ph(2^{j-1}z) = (2t)^\alpha \, e^{-(2t-1)2^{j-1}z},
\]

then $g_{j,t}(z) = e^{-rz}$ with $r:= (2t-1)2^{j-1} >0$, hence $g_{j,t}\in \{\tau\mal e^{-r\mal} \,|\, \tau\in [1,4^\alpha],r>0\}=:\mF$. Let
\[
h_{j,t}(z) := \ph(2^{j+1}z)/\ph(2^{j}tz) = {2/t}^{\alpha} \, e^{-(2-t)2^{j}z},
\]

then also $h_{j,t}\in \mF$. Since $\mF$ clearly fulfills condition (\ref{cor-glm-mE-eq}) from Corollary \ref{cor-glm-mE}, this shows that condition (UE) is fulfilled with $d=-1$, hence $\ph\in \Phi^\Sigma_{\sigma,0}$.
\end{proof}
\end{bspe}

We can now turn to the central equivalence of continuous and discrete versions of $s$-power function norms.
\begin{satz} \label{satz-s-norm-equivalence-diskret}
Let $s\in [1,\infty]$ and $A$ be an $\mR_s$-sectorial operator in $X$. Let $\sigma>\omega_{\mR_s(A)}$ and $\ph\in\Phi^\Sigma_{\sigma,0}$. Then there is a constant $C>0$ such that for all $x\in X$:
\begin{equation*} 
C^{-1} \, \Big\| \Big( \sum_{j\in \Z} |\ph(2^j A)x|^s \Big)^{1/s} \Big\|_X  \le \bigg\| \Big( \int_0^\infty |\ph(tA)x|^s \,\frac{dt}{t} \Big)^{1/s}\bigg\|_X \le  C \, \Big\| \Big(  \sum_{j\in \Z} |\ph(2^j A)x|^s\Big)^{1/s} \Big\|_X .
\end{equation*}

(with the usual modification if $s=\infty$.)
\end{satz}

\begin{proof}
We assume first that $s<\infty$. Choose the integer $d\in\Z$ due to property (UE) of $\ph$ and define
\[
S_j(t) := \big(\ph(2^{j+d}\mal)/\ph(2^{j}t\mal) \big)(A) = \ph(2^{j}tA)^{-1}\ph(2^{j+d}A)
\]

and
\[
\overline{S}_j(t) := \big( \ph(2^{j+d}t\mal)/\ph(2^{j}\mal)\big)(A)=\ph(2^{j}A)^{-1}\ph(2^{j+d}tA)
\]

for all $j\in\Z, t\in[0,1]$. By Lemma \ref{lemma_operatoren-Sj(t)-sind-Rs} the set $\mS:=\{ S_j(t), \overline{S}_j(t) \,|\, j\in\Z, t\in [1,2]\}$ is $\mR_s$-bounded, so let $C:=\mR_s(\mS)$.\\

Let $x\in X$. Define
\[
S(r):= \sum_{j\in\Z} \1I_{[2^j,2^{j+1})}(r) \ph(rA)^{-1}\ph(2^{j+d}A) \quad\text{and } y(t):=\ph(tA)x \quad\text{for all } t,r>0,
\]

then $y:(0,\infty)\to X$ is measurable, $S:(0,\infty)\to L(X)$ is strongly measurable and $S(2^jt) = \ph(2^jtA)^{-1}\ph(2^{j+d}A) = S_j(t)$ for all $t\in[1,2)$ and $j\in\Z$. Moreover, $S((0,\infty))\tm \mS$ is $\mR_s$-bounded. By Proposition \ref{prop_Rs-stetige-Fassung} we obtain
\begin{eqnarray*}
&&\Big\| \Big( \sum_{j\in \Z} |\ph(2^j A)x|^s \Big)^{1/s} \Big\|_X = \Big\| \Big( \sum_{j\in \Z} |\ph(2^{j+d} A)x|^s \Big)^{1/s} \Big\|_X  \\
&=& \Bigg\| \Bigg( \sum_{j\in \Z}  \int_1^2 |\ph(2^jtA)S_j(t)x|^s \, dt \Bigg)^{1/s} \Bigg\|_X
\approx \Bigg\| \Bigg( \sum_{j\in \Z}  \int_1^2 | S(2^jt) y(2^jt)|^s \, \frac{dt}{t} \Bigg)^{1/s} \Bigg\|_X \\
&=& \Bigg\| \Bigg( \sum_{j\in \Z}  \int_{2^j}^{2^{j+1}} |S(t)y(t)|^s \, \frac{dt}{t} \Bigg)^{1/s} \Bigg\|_X
= \Bigg\| \Bigg( \int_{0}^\infty |S(t)(t^{-s}y(t))|^s \, dt \Bigg)^{1/s} \Bigg\|_X \\
&\le& C\mal  \Bigg\| \Bigg( \int_{0}^\infty |\ph(tA)x|^s \, \frac{dt}{t} \Bigg)^{1/s} \Bigg\|_X.
\end{eqnarray*}

We now turn to the inverse inequality. By the Fatou property we obtain in a first step
\begin{eqnarray*}
&& \Big\| \Big( \int_{0}^\infty |\ph(tA)x|^s \, \frac{dt}{t} \Big)^{1/s} \Big\|_X \approx \Big\| \Big( \sum_{j\in \Z}  \int_1^2 |\ph(2^j t A)x|^s \, dt \Big)^{1/s} \Big\|_X \\
&=&  \Big\| \Big( \sum_{j\in \Z}  \int_1^2 |  \overline{S}_j(t) \ph(2^j A)x|^s \, dt \Big)^{1/s} \Big\|_X
\le \liminf_{N\to\infty}  \Big\| \Big(  \sum_{j= -N}^N  \int_1^2 |  \overline{S}_j(t) \ph(2^j A)x|^s \, dt \Big)^{1/s} \Big\|_X.
\end{eqnarray*}

Since $t\mapsto \overline{S}_j(t)$ is analytic, we can work with a version with analytic, hence in particular continuous, paths (cf. Subsection \ref{subsection_B.f.s.}) and obtain
\[
\int_1^2 |  \overline{S}_j(t) \ph(2^j A)x|^s \, dt = \lim_{\ell\to\infty} \frac1\ell  \sum_{k=1}^\ell |\overline{S}_j\big(1+\frac{k}{\ell}\big) \ph(2^j A)x|^s
\]

$\mu$-a.e. in $\Omega$. Let $ \overline{S}_{j,k}^{(\ell)} :=  \overline{S}_j\big(1+\frac{k}{\ell}\big)$ for all $j\in\Z,\ell\in\N$ and $k\in\N_{\le \ell}$, then using the Fatou property again leads to
\begin{eqnarray*}
&& \Big\| \Big( \int_{0}^\infty |\ph(tA)x|^s \, \frac{dt}{t} \Big)^{1/s} \Big\|_X \le  \liminf_{N,\ell\to\infty} \,\ell^{-1/s}\, \Big\| \Big(  \sum_{j= -N}^N  \sum_{k=1}^\ell  | \overline{S}_{j,k}^{(\ell)} \ph(2^j A)x|^s \, \Big)^{1/s} \Big\|_X\\
&\le& C\mal\liminf_{N,\ell\to\infty}  \,\ell^{-1/s}\, \Big\| \Big(  \sum_{j= -N}^N  \sum_{k=1}^\ell  |\ph(2^j A)x|^s \, \Big)^{1/s} \Big\|_X \le C\mal  \Big\| \Big(  \sum_{j= -\infty}^\infty  | \ph(2^j A)x|^s \, \Big)^{1/s} \Big\|_X.
\end{eqnarray*}

Now let $s=\infty$. Then we have trivially $\sup_{j\in\Z} |\ph(2^jA)x| \le \sup_{t>0} |\ph(tA)x|$ for all $x\in X$, hence we obtain the first inequality $\big\| \sup_{j\in\Z} |\ph(2^jA)x| \big\|_X \le \big\| \sup_{t>0} |\ph(tA)x|\big\|_X$ for all $x\in X$. For the second estimate we use the same notations as above and obtain by Proposition \ref{prop_Rs-stetige-Fassung,s=oo}
\begin{eqnarray*}
&&\big\| \sup_{t>0} |\ph(tA)x|\big\|_X = \big\| \sup_{j\in\Z}\sup_{t\in[1,2]} |\ph(2^{j+d}tA)x|\big\|_X = \big\| \sup_{j\in\Z}\sup_{t\in[1,2]} |\overline{S}_j(t) \ph(2^jA)x|\big\|_X  \\
&=&   \big\| \sup_{(j,t)\in\Z\times [1,2]}  |\overline{S}_j(t) \ph(2^jA)x|\big\|_X \le C\mal \big\| \sup_{(j,t)\in\Z\times [1,2]}  |\ph(2^jA)x|\big\|_X = C\mal\big\| \sup_{j\in\Z} |\ph(2^jA)x| \big\|_X.
\end{eqnarray*}
\end{proof}

If we combine Proposition \ref{satz-s-norm-equivalence-diskret} with Proposition \ref{satz-norm-equiv-und-Hoo} we obtain
\begin{satz} \label{satz-norm-equiv-und-Hoo-diskret}
Let $s\in [1,\infty]$ and $A$ be an $\mR_s$-sectorial operator in $X$. Let $\sigma>\omega_{\mR_s}(A)$ and $\ph,\psi\in \Phi^\Sigma_{\sigma,0}$. Then there is a constant $C>0$ such that for all $f\in H^\infty(\Sigma_\sigma)$ and $x\in X$ we have
\begin{eqnarray} \label{norm-equiv-und-Hoo-diskret}
\Big\| \Big( \sum_{j\in \Z} |f(A)\ph(2^j A)x|^s \Big)^{1/s} \Big\|_X  &\le& \  C \, \|f\|_{\infty,\sigma} \, \Big\| \Big(  \sum_{j\in \Z} |\psi(2^j A)x|^s\Big)^{1/s} \Big\|_X
\end{eqnarray}

(with the usual modification if $s=\infty$).
\qed\end{satz}

\begin{bem} \label{bem_diskrete-Normen-Haase}
By simple modifications of the proof one can show that the conclusion of Proposition \ref{satz-s-norm-equivalence-diskret} is also true in the case that $A$ is assumed to be an injective sectorial operator with dense domain and range in an arbitrary Banach space $X$, and one replaces the $s$-power function norms by the related norm expressions
\begin{equation}
\bigg( \int_0^\infty \|\ph(tA)x\|_X^s \, \frac{dt}{t} \bigg)^{1/s}, \qquad\text{and } \; \Big( \sum_{j\in\Z} \|\ph(2^jA)x\|_X^s\Big)^{1/s}, \;\text{respectively.}
\end{equation}

Moreover, also Proposition \ref{satz-norm-equiv-und-Hoo-diskret} still holds in this setting, if one uses \cite{haase}, Theorem 6.4.2 instead of Proposition \ref{satz-norm-equiv-und-Hoo}.
\end{bem}

\subsection{$\mR_s$-bounded \HU-calculus} \label{subsection_Rs-bd-Hoo}

For this subsection we fix some $s\in [1,\infty]$.

\begin{definition}
Let $\sigma>\omega(A)$. We say that $A$ has an \emph{$\mR_s$-bounded $H^\infty(\Sigma_\sigma)$-calculus} if the set
\[
\{ f(A) \,|\, f\in H^\infty(\Sigma_\sigma), \|f\|_\infty \le 1\}
\]

is $\mR_s$-bounded, which is equivalent to the existence of a constant $C>0$ such that the estimate
\[
\| (f_j(A)x_j)_j \|_{X(\ell^s)} \le C\, \sup_{j\in\N}\|f_j\|_{\infty,\sigma}\,\mal \| (x_j)_j \|_{X(\ell^s)}
\]

holds for all $f_j\in H^\infty(\Sigma_\sigma)$ and $x_j\in X$, $j\in\N$. In this case we define
\[
M^\infty_{s,\sigma}(A) :=  \mR_s\big( \big\{   f(A) \;|\; f \in H^\infty(\Sigma_\sigma), \|f\|_{\infty,\sigma} \le 1 \big\}\big).
\]

Moreover,
\[
\omega_{\mR_s^\infty}(A):= \inf\{\sigma\in (\omega(A),\pi] \:|\: \text{ $A$ has an $\mR_s$-bounded $H^\infty(\Sigma_\sigma)$-calculus}\}
\]

is called the \emph{$\mR_s^\infty$-type of $A$}, and in this situation we will also just say that $A$ has an $\mR_s$-bounded $H^\infty$-calculus.
\end{definition}

We trivially have the following
\begin{bem}
Let $A$ have an $\mR_s$-bounded \HU-calculus, then $A$ is also $\mR_s$-sectorial with $\omega_{\mR_s}(A)\le \omega_{\mR_s^\infty}(A)$.
\qed\end{bem}

The property of $A$ having an $\mR_s$-bounded \HU-calculus can be expressed in terms of the diagonal operator $\widetilde{A}_s$.

\begin{lemma} \label{Lemma-Rs-Diagonalop}
Let $\sigma,\sigma'>\omega(A)$. Consider the following assertions:
\begin{itemize}
\item[(1)] For each $f\in H^\infty(\Sigma_{\sigma})$ the operator $f(A)$ is $\mR_s$-bounded,
\item[(2)] The diagonal operator $\widetilde{A}_s$ is sectorial with $\omega(\widetilde{A}_s)<\sigma'$ and has a bounded $H^\infty(\Sigma_{\sigma'})$-calculus in $X(\ell^s)$.
\end{itemize}

Then (2) $\imp$ (1) if $\sigma \ge \sigma'$, and (1) $\imp$ (2) if $\sigma' >\sigma$. Moreover, if (1) holds there is a constant $C_\sigma>0$ such that
\[
\mR_s(f(A)) \le C_{\sigma} \mal \|f\|_{\infty,\sigma} \quad\text{for all } f\in H^\infty(\Sigma_{\sigma'})
\]

for each $\sigma'>\sigma$.
\end{lemma}

\begin{proof}
It is trivial that (2) implies (1)  if $\sigma \ge \sigma'$, so we assume  $\sigma' >\sigma$ and that (1) holds. Observe that $A$ has in particular a bounded $H^\infty(\Sigma_{\sigma})$-calculus, hence
\[
\Phi_A: H^\infty(\Sigma_{\sigma}) \to L(X), f\mapsto f(A)
\]

is bounded. By (1) we have in addition $R(\Phi_A)\tm  \RsL(X) \into L(X)$ and $\RsL(X)$ is a Banach space by Proposition \ref{RsL(X)-BS}, hence the Closed Graph Theorem implies that $\Phi_A: H^\infty(\Sigma_{\sigma}) \to \RsL(X), f\mapsto f(A)
$ is bounded, i.e. there is a constant $C_{\sigma}>0$ such that
\begin{equation} \label{gl_Lemma-Rs-Diag-1}
\mR_s(f(A)) \le C_{\sigma} \mal \|f\|_{\infty,\sigma} \quad\text{for all } f\in H^\infty(\Sigma_{\sigma}).
\end{equation}

Choose $\omega\in (\sigma,\sigma')$, then by (\ref{gl_Lemma-Rs-Diag-1}) the set $\{ \la R(\la,\widetilde{A}_s) \;|\; \la \in \C\ohne \overline{\Sigma}_{\omega}\}$ is bounded in the space $L(X(\ell^s))$, hence the diagonal operator $\widetilde{A}_s$ is sectorial with $\omega(\widetilde{A}_s)\le \omega <\sigma'$, and again (\ref{gl_Lemma-Rs-Diag-1}) implies that  the diagonal operator $\widetilde{A}_s$   has a bounded $H^\infty(\Sigma_{\sigma'})$-calculus in $X(\ell^s)$.
\end{proof}

Observe that the restriction $\sigma' >\sigma$ in (1) $\imp$ (2) if of Lemma \ref{Lemma-Rs-Diagonalop} is due to the fact that we do not assume $A$ to be $\mR_s$-sectorial. If we do this, we get the following slightly sharper condition, which can be proven in the same way.
\begin{lemma} \label{Lemma-Rs-Diagonalop-Rs-Version}
Let $A$ be an $\mR_s$-sectorial operator and $\sigma>\omega_{\mR_s}(A)$. Then the following conditions are equivalent:
\begin{itemize}
\item[(1)] For each $f\in H^\infty(\Sigma_{\sigma})$ the operator $f(A)$ is $\mR_s$-bounded,
\item[(2)] The diagonal operator $\widetilde{A}_s$ has a bounded $H^\infty(\Sigma_{\sigma})$-calculus in $X(\ell^s)$.
\end{itemize}\qed
\end{lemma}

A useful tool to check if an operator has an $\mR_s$-bounded \HU-calculus ist the following fact: Under suitable geometric conditions on $X$, $\mR_s$-boundedness of the \emph{single} operators $f(A)$ for $f\in H^\infty(\Sigma_\sigma)$ as in Lemma \ref{Lemma-Rs-Diagonalop} (1) already implies an $\mR_s$-bounded $H^\infty(\Sigma_{\sigma'})$-calculus for all $\sigma'>\sigma$:

\begin{theorem} \label{satz-Rs-beschr-Hoo-Kalkuel}
Let $\sigma,\sigma'>\omega(A)$ and $s\in [1,\infty)$, and assume that $X$ is $r$-concave for some $r<\infty$. Consider the following assertions:
\begin{itemize}
\item[(1)] $A$ has an $\mR_s$-bounded $H^\infty(\Sigma_{\sigma'})$-calculus.
\item[(2)] For each $f\in H^\infty(\Sigma_\sigma)$ the operator $f(A)$ is $\mR_s$-bounded,
\item[(3)] For each $\ph\in H_0^\infty(\Sigma_\sigma)$ the operator $\ph(A)$ is $\mR_s$-bounded, and there is a constant $C>0$ such that
\[
\forall\, \ph\in H_0^\infty(\Sigma_\sigma)\,:\; \mR_s(\ph(A)) \le C\, \|\ph\|_{\infty,\sigma}
\]
\end{itemize}

Then (1)$\imp$(3)$\imp$(2) if $\sigma\ge\sigma'$, and (2)$\imp$(1) if $\sigma'>\sigma$.\\

More precisely, if (2) holds, then for each $\omega>\sigma$ there is a constant $C_{\omega,\sigma}>0$ independent of $A$ such that
\begin{equation} \label{satz-Rs-beschr-Hoo-Kalkuel-Absch-Moo_s-gegen-sup...}
\forall\, \sigma'\ge \omega: M^\infty_{s,\sigma'}(A) \le C_{\omega,\sigma}\mal \sup\{\mR_s(f(A)) \;|\;  f\in H^\infty(\Sigma_\sigma), \|f\|_{\infty,\sigma}\le 1\}.
\end{equation}
\end{theorem}

This has been proven in \cite{artikel-Rs-Hoo} for the case $X=L^p$, and with little modifications the proof given there also works in this more general situation.\\

We conclude with the standard example:
\begin{satz} \label{satz_Laplace-Op-hat-Rs-Hoo}
Let $d,m\in\N$ and $p,s\in (1,\infty)$. Then the Laplace operator $A:=(-\Delta)^m$ has an $\mR_s$-bounded $H^\infty$-calculus in $L^p(\R^d)$ with  $\omega_{\mR_s^\infty}\big((-\Delta)^m\big)=0$.
\end{satz}

\begin{proof}
We just have to show that the operator $A\tensor \Id_{\ell^s}$ extends to a closed operator which has a bounded $H^\infty(\Sigma_\sigma)$-calculus in $L^p(\R^d,\ell^s)$ for each $\sigma>0$. Since $\ell^s$ is a UMD-spaces, this is a consequence of the vector-valued Mikhlin Multiplier Theorem. This is shown in detail for $m=1$ in \cite{levico}, Example 10.2 b).
\end{proof}

Indeed, very large classes of differential operators have an $\mR_s$-bounded \HU-calculus in $L^p$, we refer to \cite{artikel-Rs-Hoo}. The central tool used there are generalized Gaussian estimates.

\section{Generalized Triebel-Lizorkin spaces}  \label{section_The-associated-s-intermediate-spaces} 

In this section we fix $s\in [1,\infty]$, and $A$ will always denote an $\mR_s$-sectorial operator in $X$ with dense domain and dense range, where $X$ is a ocmplex Banach function space with absolut continuous norm.\\

We turn to the central topic of this paper and introduce associated homogeneous and inhomogeneous $s$-intermediate spaces $\Xdot^\theta_{s,A}, X^\theta_{s,A}$.

\subsection{Definition end elementary properties of the spaces $X^\theta_{s,A}$ and $\Xdot^\theta_{s,A}$} \label{subsection_Element-prop-of-s-spaces}

For each $\sigma\in (0,\pi]$ and $\theta\in \R$ let
\[
\Phi_{\sigma,\theta} := \{ \ph\in \mE(\Sigma_\sigma)\ohne\{0\} \,|\, z\mapsto z^{-\theta} \ph(z) \in H_0^\infty(\Sigma_\sigma) \}.
\]

For the discrete counterparts we define the subset
\[
\Phi_{\sigma,\theta}^\Sigma := \{ \ph\in \Phi_{\sigma,\theta} \,|\, z\mapsto z^{-\theta} \ph(z) \in \Phi_{\sigma,0}^\Sigma \}.
\]

Note that $\Phi_{\sigma,\theta} \into \Phi_{\sigma',\theta'}$ and $\Phi^\Sigma_{\sigma,\theta} \into \Phi^\Sigma_{\sigma',\theta'}$  for $\sigma\ge\sigma'$ and $\theta \ge \theta'$.\\

\begin{definition}
Let $\theta\in \R$, $s\in [1,\infty]$ and $\sigma>\omega(A)$. For $\ph\in \Phi_{\sigma,\theta}$ we define the corresponding $s$-power function norm as
\begin{equation} \label{theta-s-norm}
\|x\|_{\theta,s,A,\ph} := \Big\| \Big( \int_0^\infty |t^{-\theta}\ph(tA)x|^s \,\frac{dt}{t} \Big)^{1/s}\Big\|_X \quad\text{ for all } x\in X
\end{equation}

(with the usual modification if $s=\infty$). Moreover, for $\ph\in \Phi^\Sigma_{\sigma,\theta}$ we define the corresponding discrete counterpart as
\begin{equation} \label{theta-s-norm-Sigma}
\|x\|^\Sigma_{\theta,s,A,\ph} := \Big\| \Big( \sum_{j\in\Z} |2^{-j\theta}\ph(2^jA)x|^s  \Big)^{1/s}\Big\|_X \quad\text{ for all } x\in X
\end{equation}

(with the usual modification if $s=\infty$).

\end{definition}

Finally we define the space
\[
X^\theta_{s,A,\ph} := \{ x\in X \,|\, \|x\|_{\theta,s,A,\ph} < \infty \}.
\]

We will show that $\|\mal\|_{\theta,s,A,\ph}$ defined by (\ref{theta-s-norm}) actually defines a norm on $X^\theta_{s,A,\ph}$: The mapping
\[
J:X^\theta_{s,A,\ph} \to X(L^s_*), x\mapsto \big(t^{-\theta}\ph(tA)x\big)_{t>0}
\]

is linear, and $\|x\|_{\theta,s,A,\ph} = \|Jx\|_{X(L^s_*)}$ for all $x\in X^\theta_{s,A,\ph}$ by definition, hence we only have to show that $J$ is injective. Let $x\in X^\theta_{s,A,\ph}$ with $Jx=0$. Define $\rho(z) := z/(1+z)^2$ and $c:=\int_0^\infty \rho(t) |\ph(t)|^2\,\frac{dt}{t}>0$, and let $\psi:= \frac1c \rho\overline{\ph}$, then $\psi\in H^\infty_0(\Sigma_\sigma)$ and $\int_0^\infty \psi(t)\ph(t)\,\frac{dt}{t}= \frac1c\,\int_0^\infty \rho(t)|\ph(t)|^2\,\frac{dt}{t}=1$. Since $dt/t$ is a translation invariant measure on the multiplicative group $(0,\infty)$ this yields
\begin{equation} \label{gl_integral-eins}
\int_0^\infty \psi(tz)\ph(tz)\,\frac{dt}{t} = 1
\end{equation}

for all $z\in (0,\infty)$, and by analytic continuation and the identity theorem for analytic functions, (\ref{gl_integral-eins}) is also true for all $z\in\Sigma_\sigma$. By functional calculus we obtain
\[
x = \int_0^\infty \psi(tA)\underbrace{\ph(tA)x}_{=0\text{ a.e.}}\,\frac{dt}{t} = 0.
\]

By the preceding section we have the important issue that the $s$-power function norm $\|\mal\|_{\theta,s,A,\ph}$ does not depend on $\ph$ in the following sense:
\begin{satz} \label{Satz-norm-equivalent-theta}
Let $\theta\in \R$, $\sigma>\omega_{\mR_s}(A)$ and $\ph,\psi\in \Phi_{\sigma,\theta}$. Then there is a constant $C>0$ such that for all $x\in D(A^\theta)$ and $f\in H^\infty(\Sigma_\sigma)$
\begin{itemize}
\item[(1)] $\ds C^{-1 } \, \|x\|_{\theta,s,A,\ph} \le \|x\|_{\theta,s,A,\psi} \le C\,  \|x\|_{\theta,s,A,\ph}$,
\item[(2)] $\ds \|f(A) x\|_{\theta,s,A,\ph} \le C\,\|f\|_{\infty} \mal \|x\|_{\theta,s,A,\ph}$.
\end{itemize}

If $\ph,\psi\in \Phi^\Sigma_{\sigma,\theta}$, then we have in addition
\begin{itemize}
\item[(3)] $\ds C^{-1 } \, \|x\|_{\theta,s,A,\ph} \le \|x\|^\Sigma_{\theta,s,A,\psi} \le C\,  \|x\|_{\theta,s,A,\ph}$,
\end{itemize}

and $(1),(2)$ also hold for the discrete counterparts $\|\mal\|^\Sigma_{\theta,s,A,\ph},\|\mal\|^\Sigma_{\theta,s,A,\psi}$ instead of $\|\mal\|_{\theta,s,A,\ph}$, ${\|\mal\|_{\theta,s,A,\psi}}$.
\end{satz}

\begin{proof}
We apply Proposition \ref{satz-norm-equiv-und-Hoo} with $\widetilde{\ph}(z) := z^{-\theta}\ph(z)$ and $\widetilde{\psi}(z) := z^{-\theta}\psi(z)$ instead of $\ph,\psi$ and choose the constant $C>0$ as given there. Let $x\in D(A^\theta)$ and $f\in H^\infty(\Sigma_\sigma)$, then:
\begin{eqnarray*}
\|f(A) x\|_{\theta,s,A,\ph} &=& \Big\| \Big( \int_0^\infty |t^{-\theta}\ph(tA)f(A)x|^s \,\frac{dt}{t} \Big)^{1/s} \Big\|_X = \Big\| \Big( \int_0^\infty |f(A)\widetilde{\ph}(tA)A^\theta x|^s \,\frac{dt}{t} \Big)^{1/s} \Big\|_X\\
&\le& C\, \|f\|_\infty \mal \Big\| \Big( \int_0^\infty |\widetilde{\psi}(tA)A^\theta x|^s \,\frac{dt}{t} \Big)^{1/s} \Big\|_X = C\, \|f\|_\infty \mal \|x\|_{\theta,s,A,\psi}
\end{eqnarray*}

(with the usual modification if $s=\infty$). This shows (1) and (2). In the same way, (3) holds by Propositions \ref{satz-s-norm-equivalence-diskret} and \ref{satz-norm-equiv-und-Hoo-diskret}.
\end{proof}

The central objects are now the following normed spaces:

\begin{itemize}
\item[(1)] $X^\theta_{s,A,\ph}$ endowed with the norm $\|\mal \|_{X^\theta_{s,A,\ph}} := \|\mal \|_X + \|\mal\|_{\theta,s,A,\ph}$ if $\theta\ge 0$,
\item[(2)] $ \Xdot^\theta_{s,A,\ph}$ as the completion of the space $X^\theta_{s,A,\ph} $ endowed with the norm $\|\mal\|_{\theta,s,A,\ph}$.
\end{itemize}

We will see in Proposition \ref{Satz-norm-equivalent-theta-2} below (based, of course, on Proposition \ref{Satz-norm-equivalent-theta}) that these spaces are independent of $\ph$ in the sense that different $\ph\in \Phi_{\sigma,\theta}$ lead to equivalent norms, hence we shall drop $\ph$ in notation, and the space $X^\theta_{s,A}$ will be called the \emph{associated inhomogeneous $s$-intermediate space}, and $\Xdot^\theta_{s,A}$ the \emph{associated homogeneous $s$-intermediate space}. Later we will also call these spaces the associated generalized Triebel-Lizorkin spaces, when it is clear that this notion is justified. Definition (1) would also make sense for $\theta < 0$, we leave out these spaces from our considerations, since they appear to be quite unnatural. In fact, even for $\theta=0$ these spaces are delicate, since they are forced to be embedded into $X$, which might not be natural, if one looks at the concrete examples of classical Triebel-Lizorkin spaces.\\

As usual, the homogeneous space is somehow closer related to the operator $A$, but has a more complicated structure, since e.g. it is in general not embedded into $X$. Nevertheless many properties of homogeneous spaces can easily be carried over to inhomogeneous spaces: if $A$ is invertible, then $X^\theta_{s,A} \cong \Xdot^\theta_{s,A}$ for $\theta>0$, see Proposition \ref{satz_hom-vs-inhom-Raeume-fuer-A-invbar} below. We thus start with a detailed study of the homogeneous spaces.\\

The following is an important density property.
\begin{satz} \label{satz_D(Am)Schnitr-R(Am)-dicht-in-X-Punkt}
Let $\theta\in\R$, $m\in\N$ with $m>|\theta|$ and $\sigma>\omega_{\mR_s}(A)$. Then $D(A^m)\cap R(A^m)$ is a dense subset in $\Xdot^\theta_{s,A,\ph}$ for all $\ph\in \Phi_{\sigma,\theta}$.
\end{satz}

\begin{proof}
We will show first that $D(A^m)\cap R(A^m)\tm \Xdot^\theta_{s,A,\ph}$. Let $x\in D(A^m)\cap R(A^m)\tm D(A^\theta)$. Choose $\eps>0$ such that $\eps< m-|\theta|$. Let $\ph(z):= \frac{z^m}{(1+z)^{2m}}$ and $\psi_{\pm}(z):= \frac{z^{m-\theta\pm\eps}}{(1+z)^{2m}}$, then $\ph\in\Phi_{\sigma_\theta}$ and $\psi\in H_0^\infty(\Sigma_\sigma)$, and
\[
t^{-\theta}\ph(tA)x = t^{\mp\eps} (tA)^{-\theta\pm\eps}\ph(tA)x = t^{\mp\eps}\psi_{\pm}(tA) \quad\text{for all } t>0.
\]

Hence we obtain (with the usual modification if $s=\infty$)
\begin{eqnarray*}
\|x\|_{\theta,s,A,\ph} &=& \bigg\| \bigg(\int_0^\infty |t^{-\theta}\ph(tA)x|^s \, \frac{dt}{t} \bigg)^{1/s} \bigg\| \\
&\lesssim&  \bigg\| \bigg(\int_0^1 |t^{\eps}\psi_-(tA)x|^s \, \frac{dt}{t} \bigg)^{1/s} \bigg\| + \bigg\| \bigg(\int_1^\infty |t^{-\eps}\psi_+(tA)x|^s \, \frac{dt}{t} \bigg)^{1/s} \bigg\|\\
&\stackrel{(*)}{\lesssim}& \bigg\| \bigg(\int_0^1 |t^{\eps}x|^s \, \frac{dt}{t} \bigg)^{1/s} \bigg\| + \bigg\| \bigg(\int_1^\infty |t^{-\eps}x|^s \, \frac{dt}{t} \bigg)^{1/s} \bigg\| = \frac{2}{s\eps}\, \|x\|_X < \infty,
\end{eqnarray*}

where we used in $(*)$ that the operator set $\{\psi_{\pm}(tA) \,|\, t>0\}$ is $\mR_s$-bounded. This shows that $D(A^m)\cap R(A^m)\tm \Xdot^\theta_{s,A,\ph}$ for the special $\ph$ we have chosen, and by Proposition \ref{Satz-norm-equivalent-theta} this is also true for arbitrary $\ph\in \Phi_{\sigma,\theta}$ since $ D(A^m)\cap R(A^m)\tm D(A^\theta)$.\\

Now we define
\[
\widetilde{X}^\theta_{s,A,\ph} := \overline{D(A^m)\cap R(A^m)}^{\|\mal\|_{X_{\theta,s,A,\ph}}} \tm  \Xdot^\theta_{s,A,\ph},
\]

then by Proposition \ref{Satz-norm-equivalent-theta} all the spaces $\widetilde{X}^\theta_{s,A,\ph}$, where $\ph \in  \Phi_{\omega_{\mR_s}(A),\theta}$, coincide and have equivalent norms. Hence $D(A^m)\cap R(A^m)$ is dense in all $\widetilde{X}^\theta_{s,A,\ph}, \ph \in  \Phi_{\omega_{\mR_s}(A),\theta}$ if it is dense for some $\ph \in  \Phi_{\omega_{\mR_s}(A),\theta}$, so we may assume that $\ph(z)= \frac{z^m}{(1+z)^{2m}}$, hence $\ph\in\Phi^\Sigma_{\sigma,\theta}$. Let $\widetilde{X}^{\theta,\Sigma}_{s,A,\ph}$ be the completion of $D(A^m)\cap R(A^m)$ with respect to the norm $\|\mal\|^\Sigma_{\theta,s,A,\ph}$. Then again by Proposition \ref{Satz-norm-equivalent-theta} (3) we also have $\widetilde{X}^{\theta}_{s,A,\ph} = \widetilde{X}^{\theta,\Sigma}_{s,A,\ph}$ with equivalent norms, so it is enough to show that $X^{\theta}_{s,A,\ph} \tm \widetilde{X}^{\theta,\Sigma}_{s,A,\ph}$.\\

Let $x\in X^\theta_{s,A,\ph}$ and define $T_n := n(n+A^{-1})^{-1} n(n+A)^{-1} = A(\frac1n+A)^{-1} n(n+A)^{-1}$ for all $n\in\N$. Let $x_n := T_n^mx \in D(A^m)\cap R(A^m)$ for all $n\in\N$, then it is well known that $x_n\to x$ in $X$ for $n\to\infty$.\\

We consider the case $s<\infty$ first. Let $\eps>0$, then since $x\in X^\theta_{s,A,\ph}$, we can choose $N\in\N$ such that
\[
\Big\| \Big( \sum_{|j| \ge N} |2^{-j\theta}\ph(2^jA)x|^s \Big)^{1/s}\Big\|_X < \eps/2.
\]

Let $K_N:=  \sum\limits_{|j|\le N} 2^{-j\theta}\|\ph(2^jA)x\|_X$, then we can choose $n_0\in\N$ such that $K_N\mal \| x_n-x\|_X<\eps/2$ for all $n\ge n_0$. Let $n\ge n_0$, then
\begin{eqnarray*}
\| x_n-x\|^\Sigma_{\theta,s,A,\ph} &\le& \Big\| \Big( \sum_{|j|\le N} |2^{-j\theta}\ph(2^jA)(x_n-x) |^s \Big)^{1/s}\Big\|_X \\
&& + \: \Big\| \Big( \sum_{|j| \ge N} | (T_n^m-\Id)2^{-j\theta}\ph(2^jA)x|^s \Big)^{1/s}\Big\|_X\\
&\stackrel{(1)}{\lesssim}_m& \Big\| \sum_{|j|\le N} 2^{-j\theta}| \ph(2^jA)(x_n-x)|  \Big\|_X  + \: \Big\| \Big( \sum_{|j| \ge N} |2^{-j\theta}\ph(2^jA)x|^s \Big)^{1/s}\Big\|_X\\
&\le&  \sum_{|j|\le N} 2^{-j\theta}\|\ph(2^jA)x\|_X \mal \|x_n-x\|_X  + \: \Big\| \Big( \sum_{|j| \ge N} |2^{-j\theta}\ph(2^jA)x|^s \Big)^{1/s}\Big\|_X < \eps
\end{eqnarray*}

for all $n\ge n_0$, where we used in (1) that $\ell^1\into \ell^s$ and that the operators $T_n^m, n\in\N$ are $\mR_s$-bounded. So we have $\|x - x_n\|^\Sigma_{\theta,s,A,\ph} \to 0$ for $n\to\infty$.\\

Now consider the case $s=\infty$. Let $\eps>0$. Then again, since $x\in X^\theta_{\infty,A,\ph}$ and $X$ has the Fatou property, we can choose $N\in\N$ such that
\[
\big\| \sup_{j\in\Z} |2^{-j\theta}\ph(2^jA)x| - \sup_{|j| \le N} |2^{-j\theta}\ph(2^jA)x| \big\|_X < \eps/2,
\]

and we can proceed as in the first case by using the estimate
\begin{eqnarray*}
\| x-x_n\|_{\theta,\infty,A,\ph} &=&  \big\| \sup_{j\in\Z} |2^{-j\theta}\ph(2^jA)x|  \big\|_X\\
&\le& \big\| \sup_{j\in\Z} |2^{-j\theta}\ph(2^jA)x| - \sup_{|j| \le N} |2^{-j\theta}\ph(2^jA)x| \big\|_X  + \big\|  \sup_{|j| \le N} |2^{-j\theta}\ph(2^jA)x| \big\|_X.
\end{eqnarray*}
\end{proof}

Of course, Proposition \ref{satz_D(Am)Schnitr-R(Am)-dicht-in-X-Punkt} implies an analogous density property for the inhomogeneous spaces:
\begin{cor} \label{cor_D(Am)Schnitr-R(Am)-dicht-in-X-theta-s-inhomogen}
Let $\theta\ge 0$, $m\in\N_{>\theta}$ and $\sigma>\omega_{\mR_s}(A)$. Then $D(A^m)\cap R(A^m)$ is a dense subset in $X^\theta_{s,A,\ph}$ for all $\ph\in \Phi_{\sigma,\theta}$.
\qed\end{cor}

With these density properties we can extend the norm estimates from Proposition \ref{Satz-norm-equivalent-theta} to the whole spaces $\Xdot^\theta_{s,A,\ph}, X^\theta_{s,A,\ph}$.
\begin{satz} \label{Satz-norm-equivalent-theta-2}
Let $\theta\in \R$, $\sigma>\omega_{\mR_s}(A)$ and $\ph,\psi\in \Phi_{\sigma,\theta}$. Then there is a constant $C>0$ such that for all $x\in X^\theta_{s,A,\ph}$ and $f\in H^\infty(\Sigma_\sigma)$
\begin{itemize}
\item[(1)] $\ds C^{-1 } \, \|x\|_{\theta,s,A,\ph} \le \|x\|_{\theta,s,A,\psi} \le C\,  \|x\|_{\theta,s,A,\ph}$,
\item[(2)] $\ds \|f(A) x\|_{\theta,s,A,\ph} \le C\,\|f\|_{\infty} \mal \|x\|_{\theta,s,A,\ph}$.
\end{itemize}

In particular, for each $\ph,\psi\in \Phi_{\sigma,\theta}$ the spaces  $\Xdot^\theta_{s,A,\ph}$ and $\Xdot^\theta_{s,A,\psi}$ have equivalent norms, and if $\ph\in \Phi^\Sigma_{\sigma,\theta}$, then also $\|\mal\|^\Sigma_{\theta,s,A,\ph}$ is an equivalent norm on $\Xdot^\theta_{s,A,\ph}$, and $(1),(2)$ also hold for the discrete counterpart $\|\mal\|^\Sigma_{\theta,s,A,\ph}$ instead of $\|\mal\|_{\theta,s,A,\ph}$.\\

Finally, if $\theta\ge 0$, then all statements are also true for the inhomogeneous spaces $X^\theta_{s,A,\ph}$ with the inhomogeneous norms $\|\mal\|_X + \|\mal\|_{\theta,s,A,\ph}$ and $\|\mal\|_X + \|\mal\|^\Sigma_{\theta,s,A,\ph}$, respectively.
\qed\end{satz}

Hence we will usually drop the $\ph$ and sometimes $A$ in our notation of the spaces $\Xdot^\theta_{s,A,\ph}, X^\theta_{s,A,\ph}$, if there is no risk of confusion. Moreover, if $\theta\in \R$, $\sigma>\omega_{\mR_s}(A)$ and $\ph\in \Phi_{\sigma,\theta}$ (or $\ph\in \Phi^\Sigma_{\sigma,\theta}$, respectively), we will write
\[
\|x\|_{\theta,s} \approx \Big\| \Big( \int_0^\infty |t^{-\theta}\ph(tA)x|^s \,\frac{dt}{t} \Big)^{1/s} \Big\|_X \quad \bigg( \|x\|_{\theta,s} \approx \Big\| \Big( \sum_{j\in\Z}  |2^{-j\theta} \ph(2^jA)x| \Big)^{1/s} \Big\|_X  \bigg)
\]

to indicate that $\|\mal\|_{\theta,s}$ is any of the equivalent norms $\|\mal\|_{\theta,s,\psi}$, $\psi\in \Phi_{\sigma,\theta}$ (or $\|\mal\|^\Sigma_{\theta,s,\psi}$, $\psi\in \Phi^\Sigma_{\sigma,\theta}$, respectively).\\

\begin{bem} \label{bem_diskrete-Normen-Haase-im-Besovraum}
It is well know that the related norm
\begin{equation}
x \mapsto \|x\|_X + \bigg( \int_0^\infty \|t^{-\theta}\ph(tA)x\|_X^s \, \frac{dt}{t} \bigg)^{1/s}
\end{equation}

(with the usual modification if $s=\infty$), where $X$ is a general Banach space and $A$ is an injective sectorial operator in $X$ with dense domain and range, are equivalent to the norm of the real interpolation space $(X,D(A))_{\theta,s}$ if $\theta\in (0,1)$, cf. \cite{haase}, Theorem 6.5.3. By Remark \ref{bem_diskrete-Normen-Haase} we obtain in a similar way as it is done above that also the discrete counterparts
\begin{equation}
x \mapsto \|x\|_X +  \Big( \sum_{j\in\Z} \|2^{-j\theta}\ph(2^jA)x\|_X^s\Big)^{1/s}
\end{equation}

(with the usual modification if $s=\infty$) define an equivalent norm for the real interpolation space $(X,D(A))_{\theta,s}$ if $\theta\in (0,1)$ and $\ph\in \Phi^\Sigma_{\sigma,\theta}$.
\end{bem}

If $\theta>0$ and $\ph\in \Phi_{\sigma,\theta}$ for some $\sigma>\omega_{\mR_s}(A)$, we observe that by $\mR_s$-boundedness of ${\{\ph(tA) \,|\, t>0\}}$ we have
\begin{eqnarray*}
\Big\| \Big(\int_1^\infty |t^{-\theta}\ph(tA)x|^s \,\frac{dt}{t} \Big)^{1/s}\Big\|_X &\lesssim&  \Big\|\Big( \int_1^\infty |t^{-\theta}x|^s \,\frac{dt}{t} \Big)^{1/s}\Big\|_X =   (\theta s)^{-1/s}  \mal \| x \|_X.
\end{eqnarray*}

This leads to the following
\begin{bem} \label{bem_norm-int_0^1-in-inhom-X}
Let $\theta>0$ and $\ph\in \Phi_{\sigma,\theta}$ for some $\sigma>\omega_{\mR_s}(A)$. Then
\[
X^\theta_{A,s} = \bigg\{ x\in X \, \Big|\,   \Big\| \Big( \int_0^1 |t^{-\theta}\ph(tA)x|^s \,\frac{dt}{t} \Big)^{1/s} \Big\|_X < \infty \bigg\},
\]

and $\ds x\mapsto \|x\|_X +  \bigg\| \Big( \int_0^1 |t^{-\theta}\ph(tA)x|^s \,\frac{dt}{t} \Big)^{1/s} \bigg\|_X$ defines an equivalent norm on $X^\theta_{A,s}$.
\qed\end{bem}

The next proposition describes some elementary embedding properties.
\begin{prop} \label{prop_elementare-Einbettungen}
Let $\theta,\theta'\in\R$ and $r\in [1,\infty]$. Then the following embeddings hold:
\begin{itemize}
\item[(1)] If $r\le s$ and $A$ is also $\mR_r$-sectorial, then $\Xdot^\theta_{r,A}\into \Xdot^\theta_{s,A}$ and $X^\theta_{r,A}\into X^\theta_{s,A}$ if $\theta\ge 0$, respectively.
\item[(2)] If $\theta'\ge \theta > 0$, then $X^{\theta'}_{A,s} \into X^\theta_{A,s}\into X$.
\end{itemize}
\end{prop}

\begin{proof}
(1) This follows immediately if we use the discrete norm representation in the spaces $\Xdot^\theta_{r,A}, \Xdot^\theta_{s,A}$ and the fact that $\ell^r\into\ell^s$.\\

(2) This is an immediate consequence of Remark \ref{bem_norm-int_0^1-in-inhom-X}.
\end{proof}

Also the homogeneous spaces $\Xdot^\theta_{s,A}$ can be embedded into some natural extrapolation spaces associated to $A$. A suitable framework is the theory of abstract extrapolation spaces as it is developed in \cite{haase}, Chapter 6.3. We will give a short summary of those parts of the theory that we need here.\\

Defining the operator $J:= A(1+A)^{-2}: X \to X$, one has $JX=D(A)\cap R(A)\into X$, and $J$ is a topological isomorphism. The operator $J$ gives rise to a scale of extrapolation spaces
\[
X_{(1)} \into X=X_{(0)} \into X_{(-1)} \into X_{(-2)} \into \cdots \into X_{(-n)} \into \cdots
\]

where $X_{(1)} := D(A)\cap R(A)$, together with a family of (compatible) isometric isomorphisms $J_n:X_{(-n)} \to X_{(-n+1)}$ such that $J_0 = J$. The algebraic inductive limit $U:= \bigcup_{n\in\N} X_{(-n)}$ is called the \emph{universal extrapolation space} corresponding to $A$. This space can be endowed with a notion of net-convergence in the following sense: Let $(x_\alpha)_{\alpha\in A}$ be a net in $U$ and $y\in U$, then
\[
x_\alpha\to y :\iff \exists\, n\in\N, \alpha_0\in A\,:\; \big(\forall\, \alpha \in A_{\ge \alpha_0}:\: y,x_\alpha\in X_{-n}\big) \, \wedge \;\|x_\alpha-y\|\to 0.
\]

Then the limit of a net in $U$ is unique, and sum and scalar multiplication are "continuous" with respect
to the so-defined notion of convergence. Since the operator $J$ is defined on each space $X_{-n}$, $n \in \N$, it can be considered as a mapping $J:U\to U$, which then is obviously surjective, whence it is an algebraic isomorphism, continuous with respect to the notion of convergence defined above.\\

In fact, the construction of the space $U$ and in particular the notion of convergence in $U$ is only an ad-hoc construction, which is suitable to make formulations easier: For example, convergence in the space $U$ is convergence in the space $X_{(-m)}$ for some $m\in\N$, and in the same manner arguments made in the space $U$ always have to be understood to be made in the space $X_{(-m)}$ for some $m\in\N$.\\

The operator $A$ can also be lifted to the scale of extrapolation spaces and the space $U$: We define
\[
A_{(-1)}:= J^{-1} AJ \quad\text{ with domain } D(A_{(-1)}) := J^{-1}D(A).
\]

Then $A$ is an injective sectorial operator in $X_{(-1)}$ that is isometrically similar to $A$. Moreover $X_{(1)} \tm D(A_{(-1)})\tm X_{(-1)}$, and $A$ is the part of $A_{(-1)}$, i.e.
\[
A=A_{(-1)} \cap (X\times X) = \{ (x,A_{(-1)}x) \,|\, x,A_{(-1)}x \in X\}.
\]

Iterating this procedure leads to a sequence of isometrically similar sectorial operators $A_{(-n)}$ in $X_{(-n)}$ where $A_{(-n)}$ is the part of $A_{(-n-1)}$ in $X_{(-n)}$. Thus $A$ can be considered as an operator on the whole space $U$.\\

The concept of functional calculus can be extended to this framework: Let $\sigma \in (\omega(A),\pi]$ and $f\in\BB(\Sigma_\sigma)$. Then the operator $f(A)$ can be considered as an operator in each $X_{(-n)}$, and we have consistency in the sense that $f (A_{(-n-1)} )|_{X_{(-n)}} = f(A_{(-n)})$ for all $n\in\N$. To be more precise, if we choose $m\in\N$ such that $\rho^mf\in \mE(\Sigma_\sigma)$, where $\rho(z)=z/(1+z)^2$, then $f(A):X_{(-n)}\to X_{(-n-m)}$ is bounded for each $n\in\N$. Hence $f(A)$ can be considered as an operator on the space $U$, and we have the following important lemma.
\begin{lemma}[\cite{haase}, Lemma 6.3.1] \label{lemma_Haase-f(A)-in-U}
Let $\sigma \in (\omega(A),\pi)$ and $f\in\BB(\Sigma_\sigma)$. Then $D(f(A))= {\{x\in X \,|\, f(A)x\in X\}}$, i.e., the operator $f(A)$ considered as an operator in $X$ is the part in $X$ of $f(A)$ considered as an operator in $U$.
\qed\end{lemma}

Finally we define for each $\alpha\in\R$ the \emph{homogeneous fractional space}
\[
\Xdot_\alpha := A^{-\alpha}X \quad\text{ endowed with the norm } \|\mal\|_{\alpha} := \|\mal\|_{\Xdot_\alpha} := \|A^\alpha \mal\|_X,
\]

where $A^{-\alpha}$ has to be understood in the sense of Lemma \ref{lemma_Haase-f(A)-in-U} and the preceding remarks.\\

We can now give a concrete description of the homogeneous $s$-intermediate spaces as subspaces of the abstract extrapolation space $U$.
\begin{prop} \label{prop_Xdot-als-TR-von-x-aus-U-mit-endlicher-Norm}
Let $\theta\in\R$. Then
\[
\Xdot^\theta_{s,A} \cong \{ x\in U \,|\, \|x\|_{\theta,s,A}<\infty\} \into A^{-\theta} X_{(-1)}
\]
\end{prop}

\begin{proof}
For  brevity we drop $A$ in notation of norms and spaces for this proof. Define the auxiliary space
\[
\widetilde{X}^\theta_{s} :=  \{ x\in U \,|\, \|x\|_{\theta,s}<\infty\}, \quad\text{endowed with the norm } \|\mal\|_{\theta,s}.
\]

We start by showing the embedding $\widetilde{X}^\theta_{s} \into A^{-\theta} X_{(-1)}$. Choose $\sigma\in (\omega_{\mR_s}(A),\pi]$ and $\ph\in \Phi_{\sigma,\theta}$. With $\widetilde{\ph}(z):= z^{-\theta}\ph(z)$ we have $\widetilde{\ph}\in\Phi_{\sigma,0}$ and
\[
\widetilde{\ph}(tA) A^\theta x =(tA)^{-\theta}\ph(tA)A^\theta x = t^{-\theta} \ph(tA)x \quad\text{for all } x\in D(A^\theta), t>0,
\]

hence $A^\theta$ is an isomorphism from $\widetilde{X}^\theta_{s}$ to $\widetilde{X}^0_{s}$ and we may assume w.l.o.g. that $\theta=0$. Let $x\in U$ with $\|x\|_{0,s,\ph}<\infty$. We argue similar as in the proof of Proposition \ref{satz-norm-equiv-und-Hoo}. We choose a function $\psi\in H_0^\infty(\Sigma_\sigma)$ such that $\int_0^\infty \ph(t)\psi(t) \frac{dt}{t} = 1$ and conclude by the same techniques as in the proof of Proposition \ref{satz-norm-equiv-und-Hoo} that
\[
\int_0^\infty \ph(tA)\psi(tA)x \,\frac{dt}{t}= x \quad\text{in } U,
\]

i.e. the integral is taken in the extrapolation space $X_{(-m)}$ for some $m\in\N$. Let $\rho(z):= z/(1+z)^2$, choose $\omega\in (\omega_{\mR_s}(A),\sigma)$ and let $\Gamma$ by the usual parametrization of $\del\Sigma_{\omega}$. Using functional calculus and Fubini-Tonelli yields
\begin{eqnarray*}
\rho(A)x &=& \int_0^\infty \rho(A)\ph(tA)\psi(tA)x \,\frac{dt}{t}=  \int_0^\infty  \, \frac{1}{2\pi i} \int_\Gamma \rho(z)\psi(tz)zR(z,A)\ph(tA)x \,\frac{dz}{z}\,\frac{dt}{t}\\
&=&  \frac{1}{2\pi i} \int_\Gamma \rho(z) \,zR(z,A)\, \underbrace{\bigg(\int_0^\infty  \psi(tz)\ph(tA)x \,\frac{dt}{t}\bigg)}_{=: u(z)} \,\frac{dz}{z}.
\end{eqnarray*}

By H\"older's inequality we have
\begin{eqnarray*}
|u(z)| &\le& \underbrace{\bigg(\int_0^\infty |\psi(tz)|^{s'}\, \frac{dt}{t} \bigg)^{1/s'}}_{C(z):=} \mal \bigg(\int_0^\infty |\ph(tA)x|^s\, \frac{dt}{t} \bigg)^{1/s},
\end{eqnarray*}

where $C:=\sup_{z\in\Gamma} C(z)<\infty$ since $\psi\in H_0^\infty(\Sigma_\sigma)$. Since also $\rho\in H_0^\infty(\Sigma_\sigma)$ and $M:= \sup_{z\in\Sigma_\omega} \|zR(z,A)\| < \infty$ we obtain $Jx=\rho(A)x\in X$, hence $x\in X_{(-1)}$, with
\begin{eqnarray*}
\|x\|_{X_{(-1)}} &=& \|Jx\|_{X} = \|\rho(A)x\|_X \le \frac{1}{2\pi} \int_\Gamma |\rho(z)| \,\|zR(z,A)\|_X \mal  \| |u(z)|\|_X \,\frac{|dz|}{|z|}\\
&\lesssim& \frac{CM}{2\pi} \int_\Gamma |\rho(z)| \,\frac{|dz|}{|z|} \mal \bigg\| \bigg(\int_0^\infty |\ph(tA)x|^s\, \frac{dt}{t} \bigg)^{1/s}\bigg\|_X \lesssim \|x\|_{X^0_{s}}
\end{eqnarray*}

as desired.\\

We now show that $\big(\widetilde{X}^\theta_{s}, \|\mal \|_{\theta,s}\big)$ is a Banach space. Again, we may assume w.l.o.g. that $\theta=0$ and choose $\ph(z):= z/(1+z)^2$ to calculate the norm in $X^0_s$. Let $(x_n)_{n\in\N} \in \big(\widetilde{X}^0_{s} \big)^\N$ be a Cauchy sequence. Then by the already proven embedding $\widetilde{X}^0_{s} \into  X_{(-1)}$ we can find an $x\in X_{(-1)}$ with $x_n\to x$ in $x\in X_{(-1)}$, hence also $\ph(tA)x_n\to\ph(tA)x$ in $ X$ for $n\to\infty$, since by the special choice of $\ph$ we have $\ph(tA)\in L(X_{(-1)},X)$. On the other hand, $(\ph(tA)x)_{t>0}$ is a Cauchy sequence in the Banach space $X(L^s_*)$, hence we can find an $F\in X(L^s_*)$ with $\ph(\mal)Ax\to F$ in $X(L^s_*)$. We may assume w.l.o.g. by  possibly choosing subsequences, that also $\ph(\mal A)x_n\to \ph(\mal A)x$ and $\ph(\mal A)x_n\to F$ pointwise a.e. for $n\to\infty$. Thus we obtain $\ph(\mal A)x = F\in X(L^s_*)$, hence $x\in \widetilde{X}^\theta_{s}$, and $\|x-x_n\|_{X^0_s} = \|\ph(\mal A)x_n - F\|_{X(L^s_*)}\to 0$ for $n\to\infty$.\\

Since $\widetilde{X}^\theta_{s}$ is a Banach space and trivially $X^\theta_{s} \tm \widetilde{X}^\theta_{s}$, we also obtain $\Xdot^\theta_{s}\tm \widetilde{X}^\theta_{s}$, and it only remains to show the other inclusion $\widetilde{X}^\theta_{s}\tm \Xdot^\theta_{s}$. But this can easily be seen by a density argument, since for sufficiently large $m\in\N$ we have again that $D(A^m)\cap R(A^m)$ is also dense in the space $\widetilde{X}^\theta_{s}$. This can be proven in the same way as it is done in the proof of Proposition \ref{satz_D(Am)Schnitr-R(Am)-dicht-in-X-Punkt}.
\end{proof}

We sketch another possible proof of the embedding $\Xdot^0_{s,A}\into X_{(-1)}$ where we use a corresponding result for the so called McIntosh-Yagi spaces from \cite{haase}, Proposition 6.4.1.:\\

With the notations of the above proof we obtain with \cite{haase}, Proposition 6.4.1 b) the estimate
\begin{eqnarray*}
\|x\|_{X_{(-1)}} &\le& C  \mal \sup_{t>0} \|\ph(tA)x\|_X.
\end{eqnarray*}

Since $A$ is sectorial, by similar arguments as used in the proof of Proposition \ref{satz-s-norm-equivalence-diskret} (cf. also Remark \ref{bem_diskrete-Normen-Haase}) we obtain
\begin{eqnarray*}
\sup_{t>0} \|\ph(tA)x\|_X &=& \sup_{j\in\Z}\sup_{t\in[1,2]} \|\ph(2^jtA)x\|_X \approx   \sup_{j\in\Z} \|\ph(2^jA)x\|_X \lesssim  \| \sup_{j\in\Z} |\ph(2^jA)x|\|_X \\
&\lesssim& \bigg\| \bigg(\int_0^\infty |\ph(tA)x|^s\, \frac{dt}{t} \bigg)^{1/s}\bigg\|_X.
\end{eqnarray*}

With the aid of Proposition \ref{prop_Xdot-als-TR-von-x-aus-U-mit-endlicher-Norm} we can deduce a close relationship between the homogeneous and the inhomogeneous spaces:
\begin{cor} Let $\theta\ge0$, then $X^\theta_{s,A} = \Xdot^\theta_{s,A}\cap X$ with equivalent norms.
\end{cor}

\begin{proof}
This follows immediately from Proposition \ref{prop_Xdot-als-TR-von-x-aus-U-mit-endlicher-Norm} since
\[
X^\theta_{s,A} = \{ x\in X \,|\, \|x\|_{\theta,s,A}<\infty \} = \Xdot^\theta_{s,A}\cap X.
\]
\end{proof}

There are more relations between homogeneous and inhomogeneous spaces if $A$ is invertible. Moreover, the inhomogeneous spaces do not change if $A$ is replaced by $A+\eps$ for some $\eps>0$. This is contained in the following
\begin{prop} \label{satz_hom-vs-inhom-Raeume-fuer-A-invbar}
Let $\eps,\theta>0$.
\begin{itemize}
\item[(1)] If $A^{-1}\in L(X)$, then $\Xdot^\theta_{s,A} \cong X^\theta_{s,A}$,
\item[(2)] $X^\theta_{s,\eps+A} \cong X^\theta_{s,A}$.
\end{itemize}

In particular we have $X^\theta_{s,A} \cong \Xdot^\theta_{s,\eps+A}$.
\end{prop}

\begin{proof}
(1) Assume that $A^{-1}\in L(X)$. Choose  $\sigma\in (\omega_{\mR_s(A)},\pi)$ and $\ph\in\Phi^\Sigma_{\sigma,\theta}$ and let $x\in X$. By \cite{haase}, Proposition 6.5.4 we obtain
\begin{eqnarray*}
\|x\|_X &\lesssim& \big\| \big( t^{-\theta} \ph(tA)x\big)_{t>0}\big\|_{L^\infty_*(X)} = \sup_{t>0} \big\|  t^{-\theta} \ph(tA)x \big\|_{X} \stackrel{(*)}{\lesssim} \sup_{j\in\Z} \big\|  |2^{-j\theta} \ph(2^jA)x|\big\|_{X}\\
&\le& \big\| \sup_{j\in\Z}  |2^{-j\theta} \ph(2^jA)x|\big\|_{X} ,
\end{eqnarray*}

where $(*)$ can be seen by analogous arguments as in the proof of Proposition \ref{satz-s-norm-equivalence-diskret} for the case $s=\infty$, where in this case we just use the sectoriality of $A$. If $s<\infty$ we can proceed with the embedding $\ell^s \into \ell^\infty$:
\begin{eqnarray*}
\|x\|_X &\lesssim& \big\| \sup_{j\in\Z}  |2^{-j\theta} \ph(2^jA)x| \big\|_{X}  \le \Big\| \Big( \sum_{j\in\Z}  |2^{-j\theta} \ph(2^jA)x| \Big)^{1/s} \Big\|_X \approx \|x\|_{\theta,s,A},
\end{eqnarray*}

since $\ph\in\Phi^\Sigma_{\sigma,\theta}$. So we obtain
\[
\|x\|_{X^\theta_{s,A}} \approx \|x\|_X + \|x\|_{\theta,s,A} \lesssim \|x\|_{\theta,s,A},
\]

hence $\|\mal\|_{\theta,s,A}$ is an equivalent norm on the Banach space $X^\theta_{s,A}$ which implies  $X^\theta_{s,A}\iso \Xdot^\theta_{s,A}$.\\

(2) Choose $\sigma\in (\omega_{\mR_s(A)},\pi)$ and $m\in\N$ with $m-1\le\theta<m$ and define $\ph(z):=\ph_m(z):= z^m /(1+z)^m$, then $\ph\in\Phi_{\sigma,\theta}$, and for all $t>0$ we obtain
\[
\ph(t^{-1}A) = t^{-m}A^m (1+t^{-1}A)^{-m} = A^m(t+A)^{-m}.
\]

We will first show the embedding $X^{\theta}_{s,A+\eps}\into X^{\theta}_{s,A}$, so let $x\in X^{\theta}_{s,A+\eps}$. Then by Remark \ref{bem_norm-int_0^1-in-inhom-X} we have
\[
\|x\|_{X^\theta_{s,A}} \approx \|x\|_X + \bigg\| \bigg( \int_0^1 |t^{-\theta}\ph(tA)x|^s \frac{dt}{t}\bigg)^{1/s} \bigg\|_X
\]

and
\begin{eqnarray*}
\bigg\| \bigg( \int_0^1 |t^{-\theta}\ph(tA)x|^s \frac{dt}{t}\bigg)^{1/s} \bigg\|_X &=& \bigg\| \bigg( \int_1^\infty |t^{\theta}\ph(t^{-1}A)x|^s \frac{dt}{t}\bigg)^{1/s} \bigg\|_X \\
&=& \bigg\| \bigg( \int_1^\infty |t^{\theta} A^m(t+A)^{-m}x|^s \frac{dt}{t}\bigg)^{1/s} \bigg\|_X \\
&=&  \bigg\| \bigg( \int_1^\infty |t^{\theta} S(t)(\eps+A)^m(t+\eps+A)^{-m}x|^s \frac{dt}{t}\bigg)^{1/s} \bigg\|_X
\end{eqnarray*}

with $S(t):= (t+\eps+A)^{m}(\eps+A)^{-m} A^m(t+A)^{-m}$, hence
\begin{eqnarray*}
S(t)&=& \big[ (t+\eps+A)(\eps+A)^{-1} A(t+A)^{-1}\big]^m = \big[ A(\eps+A)^{-1} (t+\eps+A)(t+A)^{-1}\big]^m \\
&=& \big[ A(\eps+A)^{-1} (1+\eps(t+A)^{-1})\big]^m = \big[ A(\eps+A)^{-1} (1+\frac{\eps}{t} \mal t(t+A)^{-1})\big]^m.
\end{eqnarray*}

Since $A$ is $\mR_s$-sectorial, the range $S([1,\infty))$ is also $\mR_s$-bounded, hence by Proposition \ref{prop_Rs-stetige-Fassung}
\begin{eqnarray*}
&& \bigg\| \bigg( \int_0^1 |t^{-\theta}\ph(tA)x|^s \frac{dt}{t}\bigg)^{1/s} \bigg\|_X = \bigg\| \bigg( \int_1^\infty |t^{\theta} S(t)(\eps+A)^m(t+\eps+A)^{-m}x|^s \frac{dt}{t}\bigg)^{1/s} \bigg\|_X \\
&\lesssim& \bigg\| \bigg( \int_1^\infty |t^{\theta} (\eps+A)^m(t+\eps+A)^{-m}x|^s \frac{dt}{t}\bigg)^{1/s} \bigg\|_X
= \bigg\| \bigg( \int_0^1 |t^{-\theta}\ph(t(A+\eps))x|^s \frac{dt}{t}\bigg)^{1/s} \bigg\|_X,
\end{eqnarray*}

and we obtain
\begin{eqnarray*}
\|x\|_{X^\theta_{s,A}} &\approx& \|x\|_X + \bigg\| \bigg( \int_0^1 |t^{-\theta}\ph(tA)x|^s \frac{dt}{t}\bigg)^{1/s} \bigg\|_X\\
&\lesssim& \|x\|_X +  \bigg\| \bigg( \int_0^1 |t^{-\theta}\ph(t(A+\eps))x|^s \frac{dt}{t}\bigg)^{1/s} \bigg\|_X
\approx \|x\|_{X^\theta_{s,A+\eps}}.
\end{eqnarray*}

We now show the reverse embedding $X^{\theta}_{s,A} \into X^{\theta}_{s,A+\eps}$, so let $x\in X^{\theta}_{s,A}$. Then for all $t>0$ we have
\begin{eqnarray*}
(A+\eps)^m(t+\eps+A)^{-m} 
&=& \sum_{k=0}^m \underbrace{\binom{m}{k}\eps^{m-k}}_{=:a_k} \mal \underbrace{t^{m-k}(t+A)^{k} \mal (t+\eps+A)^{-m}}_{=: S_k(t)} \mal t^{-(m-k)}A^{k}(t+A)^{-k},
\end{eqnarray*}

and
\begin{eqnarray*}
S_k(t) &=& t^{m-k}(t+\eps+A)^{-(m-k)} \mal (t+A)^{k}(t+\eps+A)^{-k}  \\
&=& \Big[ \frac{t}{t+\eps} \mal (t+\eps)(t+\eps+A)^{-1}\Big]^{m-k} \mal \Big[\frac{t}{t+\eps} \mal (t+\eps)(t+\eps+A)^{-1} + A(t+\eps+A)^{-1} \Big]^k.
\end{eqnarray*}

This shows that also $\{S_k(t) \,|\, t>0\}$ is $\mR_s$-bounded for each $k\in(\N_0)_{\le m}$, hence
\begin{eqnarray*}
&& \bigg\| \bigg( \int_0^1 |t^{-\theta}\ph(t(A+\eps))x|^s \frac{dt}{t}\bigg)^{1/s} \bigg\|_X  
= \bigg\| \bigg( \int_1^\infty |t^{\theta} (A+\eps)^m(t+\eps+A)^{-m}x|^s \frac{dt}{t}\bigg)^{1/s} \bigg\|_X\\
&=& \bigg\| \bigg( \int_1^\infty |t^{\theta} \sum_{k=0}^m a_k \mal S_k(t) \mal t^{-(m-k)}A^{k}(t+A)^{-k}x|^s \frac{dt}{t}\bigg)^{1/s} \bigg\|_X\\
&\lesssim& \sum_{k=0}^m \bigg\| \bigg( \int_1^\infty |t^{\theta-(m-k)} \mal S_k(t) \mal A^{k}(t+A)^{-k}x|^s \frac{dt}{t}\bigg)^{1/s} \bigg\|_X\\
&\lesssim& \sum_{k=0}^m \bigg\| \bigg( \int_1^\infty |t^{\theta-(m-k)}  \mal A^{k}(t+A)^{-k}x|^s \frac{dt}{t}\bigg)^{1/s} \bigg\|_X\\
&=& \sum_{k=0}^m \bigg\| \bigg( \int_0^1 |t^{-(\theta-k)}  \ph_{m-k}(tA)x|^s \frac{dt}{t}\bigg)^{1/s} \bigg\|_X \\
&\lesssim& \sum_{k=0}^{m-2} \|x\|_{\theta-k,s,A} + \bigg\| \bigg( \int_0^1 |t^{-(\theta-(m-1))}  \ph_{1}(tA)x|^s \frac{dt}{t}\bigg)^{1/s} \bigg\|_X + \bigg\| \bigg( \int_0^1 |t^{m-\theta} x|^s \frac{dt}{t}\bigg)^{1/s} \bigg\|_X.
\end{eqnarray*}

By the choice of $m$ we have $\alpha:=m-\theta>0$ and $\beta:=\theta-(m-1)\in [0,1)$. If $m=1$, then $\beta=\theta\in (0,1)$, and we define $\delta :=\beta$. If $m\ge 2$, then $\beta\in [0,1\wedge \theta)$, hence we can choose $\delta\in (\beta,1\wedge\theta)$, and in both cases we obtain $\delta \in (0,1) \cap [\beta,\theta]$ . Then we can continue the estimate to
\begin{eqnarray*}
&& \bigg\| \bigg( \int_0^1 |t^{-\theta}\ph(t(A+\eps))x|^s \frac{dt}{t}\bigg)^{1/s} \bigg\|_X  \\
&\stackrel{\delta\ge\beta}{\lesssim}& \sum_{k=0}^{m-2} \|x\|_{\theta-k,s,A} + \bigg\| \bigg( \int_0^1 |t^{-\delta}  \ph_{1}(tA)x|^s \frac{dt}{t}\bigg)^{1/s} \bigg\|_X + \bigg\| \bigg( \int_0^1 |t^{\alpha} x|^s \frac{dt}{t}\bigg)^{1/s} \bigg\|_X\\
&\lesssim& \sum_{k=0}^{m-2} \|x\|_{\theta-k,s,A} + \|x\|_{\delta,s,A} +  (\alpha s)^{-1}\, \|x\|_X,
\end{eqnarray*}

where we used $\ph_{m-k}\in\Phi_{\sigma,\theta-k}$ for $k\in (\N_0)_{\le m-2}$ and $\ph_1 \in \Phi_{\sigma,\delta}$. So we also have an estimate
\begin{eqnarray*}
\|x\|_{X^\theta_{s,A+\eps}} &\approx& \|x\|_X + \bigg\| \bigg( \int_0^1 |t^{-\theta}\ph(t(A+\eps))x|^s \frac{dt}{t}\bigg)^{1/s} \bigg\|_X  \\
&\lesssim& \|x\|_X +   \sum_{k=0}^{m-2} \|x\|_{\theta-k,s,A} + \|x\|_{\delta,s,A}.
\end{eqnarray*}

By Proposition \ref{prop_elementare-Einbettungen} we have embeddings $X^{\theta}_{s,A} \into X^{\theta-k}_{s,A}$ for all $k\in(\N_0)_{\le m-2}$ and $X^\theta_{s,A} \into X^\delta_{s,A}$ by our choice of $\delta$, hence also
\begin{eqnarray*}
\|x\|_{X^\theta_{s,A+\eps}} &\lesssim& \sum_{k=0}^{m-2} \|x\|_{X^{\theta-k}_{s,A}} + \|x\|_{X^\delta_{s,A}} \lesssim \|x\|_{X^{\theta}_{s,A}}.
\end{eqnarray*}

\end{proof}

We now consider, as a standard example, the Laplacian in the space $L^p(\R^d)$.

\begin{satz} \label{satz_s-Raum-von-Laplace=Fspq}
Let $m,d\in\N$ and $p,s \in (1,\infty)$, and let $A:= (-\Delta)^m$ be the $m$-th power of the Laplace operator in $L^p(\R^d)$ with domain $D(A)= W^{2m,p}(\R^d)$. Let $\theta\in \R$, then
\[
\Xdot^\theta_{s,A} = \dot{F}^{2m\theta}_{p,s}(\R^d),
\]

and if $\theta>0$, then also
\[
X^\theta_{s,A} = F^{2m\theta}_{p,s}(\R^d)
\]

with equivalent norms.
\end{satz}

\begin{proof}
Choose $\sigma\in (0,\pi/2)$ and $k\in\N_{>|\theta|}$, and define $\ph(z):=z^{k} e^{-{z}^{1/m}}$ for all $z\in\Sigma_\sigma$. Then $\ph\in \Phi^\Sigma_{\sigma,\theta}$, hence $\ph$ is suitable to calculate the norm in $\Xdot^\theta_{s,A}$, and also in $X^\theta_{s,A}$ in the case $\theta>0$. On the other hand, if $t>0$ and $r:=t^{1/m}$, then
\[
\ph(tA)u = t^k (-\Delta)^{km} e^{(t^{1/m}\Delta)}u = (-r \Delta)^{km} e^{r\Delta}u  \quad\text{ for all } u\in\mS_d'.
\]

Thus \cite{triebel-1} Corollaries 3.3, 3.4, show that $\|\mal\|_{\theta,s,A,\ph}$ is also an equivalent norm for the homogeneous Triebel-Lizorkin space $\dot{F}^{2m\theta}_{p,s}$, and in the case $\theta>0$, the norm $\|\mal\|_{X^\theta_{s,A,\ph}}$ is an equivalent norm for the inhomogeneous Triebel-Lizorkin space $F^{2m\theta}_{p,s}$. Hence the result for the inhomogeneous spaces is immediate, and for the homogeneous spaces it follows from density, since $D(A^k)\cap R(A^k)\tm \dot{F}^{2m\theta}_{p,s}(\R^d)$ is a dense subspace of $\dot{F}^{2m\theta}_{p,s}(\R^d)$, and on the other hand it is also dense in $\Xdot^\theta_{s,A}$ by Proposition \ref{satz_D(Am)Schnitr-R(Am)-dicht-in-X-Punkt}.
\end{proof}

As already announced, Proposition \ref{satz_s-Raum-von-Laplace=Fspq} justifies to call the $s$-intermediate spaces the {\em generalized Triebel-Lizorkin spaces} associated to $A$.\\

Let us finally mention the correspondence of the $s$-intermediate spaces for $\mR_2$-sectorial operators to the so called Ra\-de\-macher interpolation spaces $\lsk X,Y\rsk_\theta$, which have been introduced in \cite{KKW06}, we refer also to \cite{suarez-weis} for the relationship with interpolation by the $\gamma$-method, and to \cite{kaltonweis-euclid}, where this interpolation method is studied in a general framework in connection with Euclidean structures. Then the same techniques as used in the proof of \cite{KKW06}, Theorem 7.4 show the following.
\begin{bem}
Let $X$ be $q$-concave and $p$-convex for some $1<p,q<\infty$, and assume that $A$ is $\mR_2$-sectorial. Then $\Xdot^\theta_{2,A} = \lsk X,\Xdot_1 \rsk_\theta$ for all $\theta\in (0,1)$ with equivalent norms.
\end{bem}

In fact, the inclusion $\Xdot^\theta_{2,A} \tm \lsk X,\Xdot_1 \rsk_\theta$ can be shown by similar arguments as in the proof of \cite{KKW06}, Theorem 7.4, p.\,782, and the other inclusion can be derived by means of duality with similar arguments as in the proof of \cite{KKW06}, Theorem 7.4, p.\,783f.\\

Moreover, the well known fact that having a bounded $H^\infty$-calculus is equivalent to square function estimates (cf. e.g \cite{CDMY}, \cite{KKW06}) can be reformulated in terms of the coincidence of $X$ with the space $\Xdot^0_{2,A}$:
\begin{remark}
Assume that $X$ is $q$-concave for some $q<\infty$ and let $A$ be an $\mR_2$-sectorial operator in $X$. Then $A$ has a bounded $H^\infty$-calculus in $X$ if and only if $X=\Xdot^0_{2,A}$ with equivalent norms, and in this case $\omega_{H^\infty}(A)=\omega_{\mR_2}(A)$.
\end{remark}

\subsection{The $s$-spaces as intermediate spaces and interpolation} \label{subsection_The-s-spaces-as-intermediate-spaces-and-interpolation}

We will show now that the spaces $X^\theta_{s,A}, \Xdot^\theta_{s,A}$ defined in the previous subsection are reasonable intermediate spaces. We will start with the following connection with the real interpolation spaces $(X,D(A^m))_{\alpha,q}$.
\begin{satz}
Let $\alpha>\theta>0$ and $1\le p\le s\le q\le \infty$.
\begin{itemize}
\item[(1)] If $X$ is $q$-concave, then $X^\theta_{s,A}\into (X,D(A^\alpha))_{\theta/\alpha,q}$,
\item[(2)] If $X$ is $p$-convex, then $(X,D(A^\alpha))_{\theta/\alpha,p} \into X^\theta_{s,A}$.
\end{itemize}
\end{satz}

\begin{proof}
We will only prove (1), since the proof of (2) can be done similarly. Choose $\sigma\in (\omega_{\mR_s(A)},\pi)$ and $\ph\in\Phi^\Sigma_{\sigma,\theta}$. Then an equivalent norm in $(X,D(A^\alpha))_{\theta/\alpha,q}$ is given by
\[
x\mapsto \|x\|_X + \big\| \big( t^{-\theta} \ph(tA)x\big)_{t>0}\big\|_{L^q_*(X)},
\]

cf. \cite{haase}, Theorem 6.5.3. We consider only the case $q<\infty$, the case $q=\infty$ can be treated similarly as usual. By analogous arguments as in the proof of Proposition \ref{satz-s-norm-equivalence-diskret} using the sectoriality of $A$ we obtain
\begin{eqnarray*}
\big\| \big( t^{-\theta} \ph(tA)x\big)_{t>0}\big\|_{L^q_*(X)} &=& \bigg(\int_0^\infty  \|t^{-\theta} \ph(tA)x\|_X^q \frac{dt}{t}\bigg)^{1/q} \approx \Big( \sum_{j\in\Z}  \|2^{-j\theta} \ph(2^jA)x\|_X^q \Big)^{1/q} \\
&\stackrel{(*)}{\le}&   M_{(q)}(X) \mal \Big\|  \Big( \sum_{j\in\Z}  |2^{-j\theta} \ph(2^jA)x|^q \Big)^{1/q} \Big\|_X \\
&\le&  M_{(q)}(X) \mal \Big\|  \Big( \sum_{j\in\Z}  |2^{-j\theta} \ph(2^jA)x|^s \Big)^{1/s} \Big\|_X \approx \|x\|_{\theta,s,A}.
\end{eqnarray*}

where in (*) we used the Fatou-property and the $q$-concavity of $X$ ($M_{(q)}(X)$ is the $q$-concavity constant of $X$), and in the last inequality we used that $\ell^s\into\ell^q$. Hence we also obtain
\[
\|x\|_{(X,D(A^\alpha))_{\theta/\alpha,q}} \approx \|x\|_X + \big\| \big( t^{-\theta} \ph(tA)x\big)_{t>0}\big\|_{L^q_*(X)} \lesssim  \|x\|_X + \|x\|_{\theta,s,A} \approx \|x\|_{X^\theta_{s,A}}.
\]

\end{proof}

In particular, since $X$ is always $\infty$-concave and $1$-convex, we trivially have
\begin{cor}  \label{cor_Raum-X^theta-zwischen-reellen-Interpol.}
Let $\alpha>\theta>0$, then
\begin{equation} \label{cor_Raum-X^theta-zwischen-reellen-Interpol.-eq}
(X,D(A^\alpha))_{\theta/\alpha,1} \into X^\theta_{s,A}\into (X,D(A^\alpha))_{\theta/\alpha,\infty}.
\end{equation}
\qed\end{cor}

We can now use standard reiteration of real interpolation spaces, which leads to the following
\begin{cor}
\begin{itemize}
\item[(1)] Let $\alpha > \theta > \beta >0$, then $D(A^\alpha)\into X^\theta_{s,A}\into D(A^\beta)$.
\item[(2)] Let $\theta_0<\theta_1<\alpha$ and $s_0,s_1\in [1,\infty]$ and $\delta\in (0,1)$, then
\begin{equation}
(X^{\theta_0}_{s_0,A}, X^{\theta_1}_{s_1,A})_{\delta,q} = (X,D(A^\alpha))_{\theta/\alpha,q} \quad\text{with } \theta := (1-\delta)\theta_0 + \delta \theta_1.
\end{equation}
\end{itemize}
\end{cor}

\begin{proof}
(1) This is simply due to the fact that by real interpolation theory and Corollary \ref{cor_Raum-X^theta-zwischen-reellen-Interpol.} we have
\[
D(A^\alpha)\into (X,D(A^\alpha))_{\theta/\alpha,1} \into X^\theta_{s,A}\into (X,D(A^\alpha))_{\theta/\alpha,\infty} \into (X,D(A^\alpha))_{\beta/\alpha,1} \into D(A^\beta).
\]

For (2) we observe that equation (\ref{cor_Raum-X^theta-zwischen-reellen-Interpol.-eq}) from Corollary \ref{cor_Raum-X^theta-zwischen-reellen-Interpol.} is equivalent to the fact that the spaces $X^{\theta_j}_{s,A}$ are in the class $J_{\theta_j/\alpha}(X,D(A^\alpha))\cap K_{\theta_j/\alpha}(X,D(A^\alpha))$ for $j=0,1$, hence (2) follows from the reiteration theorem for real interpolation.
\end{proof}

Next we consider complex interpolation of the $s$-spaces.

\begin{prop} \label{satz_komplexe-Interpol-Xs-Raeume}
Let $s_0,s_1\in (1,\infty)$ and assume that $A$ has an $\mR_{s_j}$-bounded \HU-calculus for $j=0,1$. Let $\theta_0,\theta_1\in\R$ with
\begin{equation} \label{gl-satz_komplexe-Interpol-Xs-Raeume-1}
|\theta_0|,|\theta_1|,|\theta_0-\theta_1| < \frac{\pi}{\omega_0},
\end{equation}

where $\omega_0:=\omega_{\mR_{s_0}^\infty}(A) \vee \omega_{\mR_{s_1}^\infty}(A) <\pi$. Then for
\[
\Big(\theta,\frac1s \Big) := (1-\alpha) \Big( \theta_0,\frac{1}{s_0} \Big) + \alpha \Big(\theta_1,\frac{1}{s_1}\Big)
\]

we have $\ds \big[ \Xdot^{\theta_0}_{s_0,A},\Xdot^{\theta_1}_{s_1,A} \big]_\alpha \iso \Xdot^{\theta}_{s,A}$, and if $\:\theta_0,\theta_1>0$ we also have $\ds \big[ X^{\theta_0}_{s_0,A},X^{\theta_1}_{s_1,A} \big]_\alpha \iso X^{\theta}_{s,A}$.
\end{prop}

We remark that the restriction (\ref{gl-satz_komplexe-Interpol-Xs-Raeume-1}) for the interpolation indices $\theta_j$ is due to our method of proof. It seems not to be clear if this restriction is reasonable or just a matter of lack of technique in our method. Observe that in the classical situation of Triebel-Lizorkin-spaces where $A=-\Delta$ we have $\omega_0=0$, hence the restriction (\ref{gl-satz_komplexe-Interpol-Xs-Raeume-1}) drops out and Proposition \ref{satz_komplexe-Interpol-Xs-Raeume} is consistent with the classical results.

\begin{proof}[Proof of Proposition \ref{satz_komplexe-Interpol-Xs-Raeume}]
First we observe that if we have proved the result for the homogeneous spaces, then by Proposition \ref{satz_hom-vs-inhom-Raeume-fuer-A-invbar} we obtain the corresponding result for the inhomogeneous spaces, hence it is sufficient to consider the homogeneous spaces.\\

According to the assumption on $\theta_0,\theta_1$ we can choose  $0<\alpha<\beta < \frac{\pi}{\omega_0}$ such that $\theta_j \in (\alpha-\beta,\alpha)$ for $j=0,1$. Furthermore fix some $\sigma \in \big(\omega_0,\frac{\pi}{\beta}\big)$.\\

We will show first that we can reduce to the case $\beta=1$. If $\beta<1$, we can obviously replace $\beta$ by $1$, so consider the case $\beta>1$. Since $A$ has an $\mR_{s_j}$-bounded \HU-calculus, the operator $A^\beta$ also has an $\mR_{s_j}$-bounded \HU-calculus with
\[
\omega_{\mR_{s_j}^\infty}(A^\beta) \le \beta \omega_{\mR_{s_j}^\infty}(A) \le \beta \omega_0 < \pi \quad\text{for } j=0,1.
\]

Observe that $\Xdot^{\beta\delta}_{s,A} \iso \Xdot^{\delta}_{s,A^\beta}$ for $\delta\in\R$ canonically, since for $\ph\in \Phi_{\beta\sigma,\delta}$ and $\psi(z):= \ph(z^\beta)$ we have $\psi\in\Phi_{\sigma,\beta\delta}$ and
\begin{eqnarray*}
\|x\|_{\beta\delta,s,A,\psi} &=& \bigg\| \bigg( \int_0^\infty |t^{-\beta\delta} \psi(tA)x|^s \,\frac{dt}{t}\bigg)^{1/s} \bigg\|_X = \bigg\| \bigg( \int_0^\infty |t^{-\beta\delta} \ph(t^\beta A^\beta)x|^s \,\frac{dt}{t}\bigg)^{1/s} \bigg\|_X\\
&=& \beta^{-1/s} \mal \bigg\| \bigg( \int_0^\infty |t^{-\delta} \ph(t A^\beta)x|^s \,\frac{dt}{t}\bigg)^{1/s}  \bigg\|_X \approx \|x\|_{\delta,s,A^\beta,\ph}.
\end{eqnarray*}

So since the above stated isomorphisms are canonical we can replace $A$ by $A^\beta$ and then $\beta$ by $1$.\\

We define auxiliary spaces
\[
\ell^{s,\theta} := \ell^{s,\theta}(\Z) := \Big\{ (\alpha_j)_j \in \C^\Z \;\big|\; \|(\alpha_j)\|_{\ell^{s,\theta}} := \big\| ( 2^{-\theta j}\alpha_j )_j \big\|_{\ell^s} < \infty \Big\},
\]

endowed with the weighted norm $\| \mal \|_{\ell^{s,\theta}}$ for all $s\in [1,\infty], \theta\in\R$. Then we have for all $s_0,s_1\in [1,\infty]$ and $\theta_0,\theta_1\in \R$ and $\alpha\in (0,1)$
\begin{equation} \label{gl_interpolation-X(ell^s,theta)}
[X(\ell^{s_0,\theta_0}), X(\ell^{s_1,\theta_1})]_\alpha \iso  X(\ell^{s,\theta})  \quad\text{where } \Big(\theta,\frac1s \Big) = (1-\alpha) \Big( \theta_0,\frac{1}{s_0} \Big) + \alpha \Big(\theta_1,\frac{1}{s_1}\Big),
\end{equation}

cf. \cite{calderon} 13.6(i) and \cite{bergh}, Theorem 5.6.3.\\[-2,2ex]

We show that the homogeneous spaces are retracts of the spaces $X(\ell^{s,\theta})$ with canonical (co-)retractions and then use \cite{triebel-interpolation}, 1.2.4 Theorem and (\ref{gl_interpolation-X(ell^s,theta)}). To this purpose we shall construct a coretraction $J:\Xdot^\theta_{s,A}\to X(\ell^{s,\theta})$ and a retraction  $P: X(\ell^{s,\theta})\to \Xdot^\theta_{s,A}$, i.e. bounded operators such that $PJ=\Id_{\Xdot^\theta_{s,A}}$, which are independent of $\theta \in (\alpha-1,\alpha)$ and  $s\in [1,\infty]$ such that $A$ has an $\mR_s$-bounded \HU-calculus with $\omega_s\le\omega_0$.\\

By the above reduction we have $\beta=1$, hence $\alpha\in (0,1)$. We define auxiliary functions
\[
\ph(z) := -\frac{z^{\alpha}}{2+z}, \quad \psi(z):=  \frac{z^{1-\alpha}}{1+z}, \quad\rho(z) := \frac{z^{\alpha}}{1+z}
\]

and $f(z):= \ph(z)\psi(z) = \frac{-z}{(1+z)(2+z)}= \frac{1}{1+z}- \frac{2}{2+z}$ for all $z\in\Sigma_\sigma$. The following estimate is easily proved: If $a,b,c>0$ and $g(z):= \frac{cz}{(1+az)(1+bz)}$, then
\begin{equation} \label{abschaetzung_von-g-fuer-(Ko)retrakt}
|g(z)| \lesssim_\sigma \frac{c}{a+b} \quad\text{for all } z\in\Sigma_\sigma.
\end{equation}

We now define the operator $Jx := \big( \ph(2^jA)x \big)_{j\in\Z}$ for all $x\in X$ and formally
\[
P(y_j)_j := \sum_{j\in\Z} \psi(2^jA)y_j := \lim_{N\to\infty} \underbrace{\sum_{j=-N}^N \psi(2^jA)y_j}_{=: P_N(y_j)_j }  \quad\text{for } (y_j)_j\in X^\Z .
\]

Let $\theta \in (\alpha-1,\alpha) $ and $s\in [1,\infty]$ such that $A$ has an $\mR_s$-bounded \HU-calculus with $\omega_s\le\omega_0$. We show that $J|_{\Xdot^\theta_{s,A}}: \Xdot^\theta_{s,A}\to X(\ell^{s,\theta})$ is a coretraction and  $P|_{X(\ell^{s,\theta})}: X(\ell^{s,\theta})\to \Xdot^\theta_{s,A}$ is a corresponding retraction, i.e. we have to show that
\begin{itemize}
\item[(1)] $J|_{\Xdot^\theta_{s,A}}: \Xdot^\theta_{s,A}\to X(\ell^{s,\theta})$ is bounded,
\item[(2)] $P|_{X(\ell^{s,\theta})}: X(\ell^{s,\theta})\to \Xdot^\theta_{s,A}$ is well-defined and bounded, and
\item[(3)] $PJx=x$ for all $x\in X^\theta_{s,A}$.
\end{itemize}

{\em Ad (1):} This is simply due to the fact that $\ph\in \Phi^\Sigma_{\sigma,\theta}$ (this can be shown similar as Example \ref{bspe_Standardfkt.-in-Phi^Sigma} (1)).\\

{\em Ad (2):} We now use the function $\rho$ to calculate the norm in the space $\Xdot^\theta_{s,A}$, which is possible since also $\rho\in  \Phi^\Sigma_{\sigma,\theta}$:
\begin{eqnarray*}
&& \|P_N(y_j)_j\|_{\Xdot^\theta_{s,A}} = \bigg\| \Big( 2^{-k\theta}\rho(2^kA)\sum_{j=-N}^N \psi(2^jA)y_j\Big)_k  \bigg\|_{X(\ell^s)} \\
&=&  \bigg\| \Big( \sum_{j=-N}^N 2^{-(k-j)\theta}\rho(2^kA)\, 2^{-j\theta}\psi(2^jA)y_j\Big)_k  \bigg\|_{X(\ell^s)}\\
&=&  \bigg\| \Big( \sum_{\ell=-k-N}^{-k+N} 2^{\ell\theta}\rho(2^kA)\psi(2^{k+\ell}A)y_{k+\ell}\Big)_k  \bigg\|_{X(\ell^s)}
\le  \bigg\| \Big( \sum_{\ell\in\Z} |2^{\ell\theta}\rho(2^kA)\psi(2^{k+\ell}A)y_{k+\ell}| \Big)_k  \bigg\|_{X(\ell^s)}\\
&\le&  \sum_{\ell\in\Z} \big\| \big( |2^{\ell\theta} \rho(2^kA)\psi(2^{k+\ell}A)y_{k+\ell}| \big)_k  \big\|_{X(\ell^s)}\\
&\lesssim& \sum_{\ell\in\Z}  \Big(\sup_{z\in\Sigma_\sigma}| 2^{\ell\theta} \rho(2^kz)\psi(2^{k+\ell}z)|\Big) \, \big\| \big( y_{k+\ell}\big)_k  \big\|_{X(\ell^s)}
\le C\,  \big\| \big( y_{k}\big)_k  \big\|_{X(\ell^s)}
\end{eqnarray*}

with $C:=\sum\limits_{\ell\in\Z}  \Big(\sup\limits_{z\in\Sigma_\sigma}| 2^{\ell\theta} \rho(2^kz)\psi(2^{k+\ell}z)|\Big)$. So we only have to show that $C<\infty$, because then with the Fatou property we obtain
\[
\|P(y_j)_j\|_{\Xdot^\theta_{s,A}} \le \liminf_{N\to\infty} \|P_N(y_j)_j\|_{\Xdot^\theta_{s,A}} \lesssim C\mal  \big\| \big( y_{k}\big)_k  \big\|_{X(\ell^s)}.
\]

For this let
\[
g(z):= \rho(2^kz)\psi(2^{k+\ell}z) = - \frac{cz}{(1+az)(1+bc)},
\]

where $a:= 2^k, b:= 2^{k+\ell}$ and $c:= a^\alpha b^{1-\alpha}= 2^{\alpha k + (1-\alpha)(k+\ell)} = 2^{ k + (1-\alpha)\ell}$, then by (\ref{abschaetzung_von-g-fuer-(Ko)retrakt})
\[
|g(z)| \lesssim \left(\frac{a+b}{c}\right)^{-1} = \left( 2^{-(1-\alpha)\ell} + 2^{\alpha\ell}\right)^{-1}\le  2^{-\alpha\ell} \wedge 2^{(1-\alpha)\ell}.
\]

For $\ell\in\N_0$ this implies  $2^{\theta\ell}|g(z)|\lesssim 2^{-(\alpha-\theta)\ell}$, and for $\ell\in -\N$ we have $2^{\theta\ell}|g(z)|\lesssim 2^{-(\theta+ 1-\alpha)|\ell |}$, hence altogether with $\delta:=\min\{\alpha-\theta, (\theta+ 1-\alpha) \}>0$:
\[
\sup_{z\in\Sigma_\sigma}| 2^{\ell\theta} \rho(2^kz)\psi(2^{k+\ell}z)| = \sup_{z\in\Sigma_\sigma}| 2^{\ell\theta} g(z) | \lesssim 2^{-\delta|\ell|} \quad\text{for all } \ell\in\Z,
\]

so $C \lesssim \sum\limits_{\ell\in\Z}  2^{-\delta|\ell|}<\infty$ as desired.\\

{\em Ad (3):} Let $x\in X^\theta_{s,A}$, then
\begin{eqnarray*}
\sum_{j=-N}^N \psi(2^jA)\ph(2^jA)x &=& \sum_{j=-N}^N f(2^jA)x =  \sum_{j=-N}^N \Big( 2^{-j}(2^{-j}+A)x - 2^{-(j-1)}(2^{-(j-1)}+A)x  \Big)\\
&=& 2^{-N}(2^{-N}+A)x - 2^{N+1}(2^{N+1}+A)x \stackrel{N\to\infty}{\rechts} 0-x=x \quad\text{in } \Xdot^\theta_{s,A},
\end{eqnarray*}

since the part of $A$ in $\Xdot^\theta_{s,A}$ is sectorial (cf. Lemma \ref{lemma_parts-in-s-Raeumen-sektoriell+Definitionsbereich} in the following subsection).\\

Since the operators $J,P$ are appropriate (co-)retractions, the claim follows by \cite{triebel-interpolation}, 1.2.4 Theorem  from (\ref{gl_interpolation-X(ell^s,theta)}).
\end{proof}

We conclude this subsection with a variant of the interpolation property for the $s$-intermediate spaces.

\begin{prop} \label{satz-interpolation-property-in-s-spaces}
Let $\theta\in\R$ and $m\in\N$ with $m>|\theta|$, and let $T\in L(X)$ be $\mR_s$-bounded. If $T(D(A^m)\cap R(A^m))\tm D(A^m)\cap R(A^m)$, then $T\in L(\Xdot^\theta_{s,A})$, and if $\theta\ge 0$ and $T(D(A^m))\tm D(A^m)$, then $T\in L(X^\theta_{s,A})$.
\end{prop}

\begin{proof}
This is an easy consequence of Proposition \ref{satz-norm-equiv-und-Hoo} and Proposition \ref{satz_D(Am)Schnitr-R(Am)-dicht-in-X-Punkt}.
\end{proof}

\subsection{The part of $A$ in its associated Triebel-Lizorkin spaces} \label{subsection_The-part-of-A-in-the-s-intermediate-spaces}

We fix some $\theta\in\R$. Recall that we can extrapolate the operator $A$ to an operator in the universal extrapolation space $U$ such that $A$ is sectorial in each extrapolation space $X_{(-m)}, m\in\N$.\\

Observe first that the operators $A^\alpha$ shift the scales of associated $s$-spaces in the following sense.
\begin{lemma} \label{lemma_A^alpha-in-Skala-der-s-Raeume}
Let $\alpha\in\R$. Then $A^\alpha X^\theta_{s,A}$ coincides with $X^{\theta-\alpha}_{s,A}$ in the set-theoretical sense, and the operator $A^\alpha$ (defined on $U$) induces a topological isomorphism
\[
A^\alpha: \Xdot^\theta_{s,A} \to \Xdot^{\theta-\alpha}_{s,A}.
\]

If in addition $\theta>\alpha \vee 0$, then also the operator $(1+A)^\alpha$ induces an isomorphism
\[
(1+A)^\alpha: X^\theta_{s,A} \to X^{\theta-\alpha}_{s,A}.
\]
\end{lemma}

\begin{proof}
Choose  $\sigma\in (\omega_{\mR_s(A)},\pi)$ and $\ph\in\Phi_{\sigma,\theta-\alpha}$ such that $\psi(z):=z^{\alpha}\ph(z)$ defines a function in $\mE(\Sigma_\sigma)$, then $\psi\in\Phi_{\sigma,\theta}$. Let $x\in X$, then
\begin{eqnarray*}
\|A^\alpha x\|_{\theta-\alpha,s,A} &\approx& \bigg\| \bigg( \int_0^\infty |t^{-\theta+\alpha}\ph(tA)A^\alpha x|^s\, \frac{dt}{t}\bigg)^{1/s} \bigg\|_X = \bigg\| \bigg( \int_0^\infty |t^{-\theta} \psi(tA)x|^s\, \frac{dt}{t}\bigg)^{1/s} \bigg\|_X\\
&\approx& \|x\|_{\theta,s,A},
\end{eqnarray*}

hence $A^\alpha x\in X^{\theta-\alpha}_{s,A} \iff x\in X^{\theta}_{s,A}$, and $A^\alpha: ( X^{\theta}_{s,A}, \|\mal\|_{\theta,s,A,\psi})\to ( X^{\theta-\alpha}_{s,A}, \|\mal\|_{\theta,s,A,\ph})$ is an isometric isomorphism. Since  $X^{\theta-\gamma}_{s,A}$ is dense in $\Xdot^{\theta-\gamma}_{s,A}$ for $\gamma\in \{0,\alpha\}$ this also yields $A^\alpha x\in \Xdot^{\theta-\alpha}_{s,A}\iff x\in \Xdot^{\theta}_{s,A}$ for all $x\in U$ and that $A^\alpha$ induces a topological isomorphism $A^\alpha: \Xdot^\theta_{s,A} \to \Xdot^{\theta-\alpha}_{s,A}$.\\[-1ex]

If in addition $\theta>\alpha \vee 0$, then $X^{\theta-\gamma}_{s,A}\iso \Xdot^{\theta-\gamma}_{s,A+1}$ for $\gamma\in \{0,\alpha\}$ by Proposition \ref{satz_hom-vs-inhom-Raeume-fuer-A-invbar}, and we can apply the first part for $1+A$ instead of $A$.
\end{proof}

\begin{definition}
Let $\dot{A}_{\theta,s}:=A_{\Xdot^\theta_{A,s}}$ be the part of $A$ in $\Xdot^\theta_{A,s}$ and  $A_{\theta,s}:=A_{X^\theta_{A,s}}$ be the part of $A$ in $X^\theta_{A,s}$ if $\theta\ge 0$, respectively.
\end{definition}

\begin{bem} \label{bem_s-Raeume-invariant-unter-Resolventen}
The spaces $\Xdot^\theta_{s,A}$, and $X^\theta_{A,s}$ if $\,\theta\ge 0$, respectively, are invariant under resolvents of $A$, i.e.  $ R(\la,A)\Xdot^\theta_{A,s}\tm  \Xdot^\theta_{A,s}$, and if $\theta\ge 0$ then $R(\la,A)X^\theta_{A,s,\ph}\tm  X^\theta_{A,s,\ph}$, respectively, for all $\la\in \C\ohne \overline{\Sigma}_\sigma$ and $\sigma\in (\omega_{\mR_s(A)},\pi)$. In fact, it is sufficient to show that $\la R(\la,A)X^\theta_{s,A}\tm X^\theta_{s,A}$. To see this we let $x\in X^\theta_{s,A}$ and choose $\sigma\in (\omega_{\mR_s(A)},\pi)$ and $\ph\in\Phi_{\sigma,\theta}$, then
\begin{eqnarray*}
\|\la R(\la,A)x\|_{\theta,s,A} &\approx& \big\| (\la R(\la,A)) t^{-\theta}\ph(tA)x)_{t>0} \big\|_{X(L^s_*)} \le  M_{\mR_s,\sigma}(A)\, \big\| ( t^{-\theta}\ph(tA)x)_{t>0} \big\|_{X(L^s_*)}\\
&\approx& \|x\|_{\theta,s,A}
\end{eqnarray*}

for all $\la\in \C\ohne \overline{\Sigma}_\sigma$, since the set $\{ z R(z,A) \:|\: z\in \C\ohne \overline{\Sigma}_\sigma\}$ is $\mR_s$-bounded.
\qed \end{bem}

By Remark \ref{bem_s-Raeume-invariant-unter-Resolventen} we obtain the following elementary properties of the operators $\dot{A}_{\theta,s}, A_{\theta,s}$.

\begin{lemma} \label{lemma_parts-in-s-Raeumen-sektoriell+Definitionsbereich}
The operator $\dot{A}_{\theta,s}$ is an injective sectorial operator in $\Xdot^\theta_{s,A}$ of type $\omega(\dot{A}_{\theta,s})\le \omega_{\mR_s}(A)$ with $D(\dot{A}_{\theta,s})=\Xdot^\theta_{s,A} \cap \Xdot^{\theta+1}_{s,A}$. If $\theta\ge 0$, the operator $A_{\theta,s}$ is an injective sectorial operator in $X^\theta_{s,A}$ of type $\omega({A}_{\theta,s})\le \omega_{\mR_s}(A)$ with $D(A_{\theta,s}) = X^{\theta+1}_{s,A}$.\\

Moreover, if $m\in\N_{>|\theta|}$, then $D(A^m)\cap R(A^m)$ is a core of $\dot{A}_{\theta,s}$, and of $A_{\theta,s}$ in the case $\theta\ge 0$, respectively.
\end{lemma}

\begin{proof}
It is well know that the statements of Remark \ref{bem_s-Raeume-invariant-unter-Resolventen}  imply the asserted sectoriality properties, so we only have to verify the statements concerning the domains. But this follows immediately from Lemma \ref{lemma_A^alpha-in-Skala-der-s-Raeume} with $\alpha=1$. The final assertions follows from the approximation result that is also used in the proof of Proposition \ref{satz_D(Am)Schnitr-R(Am)-dicht-in-X-Punkt}.
\end{proof}

Combining Lemma \ref{lemma_parts-in-s-Raeumen-sektoriell+Definitionsbereich} with Proposition \ref{Satz-norm-equivalent-theta-2} we immediately obtain the final main result of this paper.
\begin{theorem} \label{Theorem_part-A-hat-Hoo-in-s-Raeumen}
\begin{itemize}
\item[(1)] The part $\dot{A}_{\theta,s}$ of $A$ in $\Xdot^\theta_{s,A}$ with domain $D(\dot{A}_{\theta,s})=\Xdot^\theta_{s,A} \cap \Xdot^{\theta+1}_{s,A}$ has a bounded \HU-calculus with $\omega_{H^\infty}(\dot{A}_{\theta,s})\le \omega_{\mR_s}(A)$.
\item[(2)] Let $\theta\ge 0$. If $A^{-1}\in L(X)$ or $A$ has a bounded \HU-calculus in $X$, then the part $A_{\theta,s}$ of $A$ in $X^\theta_{s,A}$ with domain $D(A_{\theta,s}) = X^{\theta+1}_{s,A}$ has a bounded \HU-calculus with $\omega_{H^\infty}(A_{\theta,s})\le \omega_{\mR_s}(A)$.
\end{itemize}
\qed\end{theorem}

As we already noted in the introduction of this section, Theorem \ref{Theorem_part-A-hat-Hoo-in-s-Raeumen} can be seen as a variant of Dore's Theorem that states that an invertible sectorial operator $A$ in a Banach space $X$ has a bounded \HU-calculus in the scale of \emph{real interpolation spaces} $(X,D(A))_{p,\theta}$ for $p\in[1,\infty], \theta\in (0,1)$, cf. \cite{dore-interpolation1} and \cite{dore-interpolation2}, and also the extensive treatment using functional calculus which is given in \cite{haase}, Chapter 6.

\section{Concluding remarks} \label{section_Concluding_remarks}

1. In the last decade, the notion of vector-valued Triebel-Lizorkin spaces $F^{\alpha}_{p,s}(E)$, where $E$ is a Banach space, has become of greater interest e.g. in connections with regularity theory for Cauchy-Problems in an $L^p$-$L^q$ setting with $p\neq q$ (cf., e.g. \cite{denk-hieber-pruess-inhomogen}). We note that the techniques developed in this article also work in vector-valued settings, i.e. one can replace the scalar-valued function space $X$  by a vector-valued function space $X(E)$. In this case one has to replace the (pointwise) absolut values by pointwise norms $|\mal|_E$. For example, the notion of $\mR_s$-boundedness has to be replaced by $\mR_s(E)$-boundedness in the following sense:
\begin{equation}
\Big\| \Big( \sum_{j=1}^n |T_jx_j|_E^s \Big)^{1/s} \Big\|_X \lesssim  \Big\| \Big( \sum_{j=1}^n |x_j|_E^s \Big)^{1/s} \Big\|_X.
\end{equation}

In the same manner, the notions of $\mR_s$-sectoriality, -bounded \HU-calculus and $s$-power-function norms have to be modified. Then all conclusion in this article are still true, except the connection of $\mR_2(E)$-boundedness and $\mR$-boundedness: This will fail in general unless $E$ is a Hilbert space. If $E \tm M(\widetilde{\Omega},\widetilde{\mu})$ is itself a Banach function space with absolut continuous norm, there is also a second possibility to adapt the theory of this article. In this case, the space $X(E)$ can be canonically identified with a scalar-valued Banach function space $XE \tm M(\Omega\times \widetilde{\Omega}, \mu\tensor \widetilde{\mu})$ via the identification of $x:\Omega\to E$ with a measurable function $x:\Omega\times \widetilde{\Omega} \to\K$, and the theory of this article can be applied by replacing the space $X$ by $XE$. In this context it is important to note that then $\mR_s(E)$-boundedness in the space $X(E)$ will in general \emph{not} coincide with $\mR_s$-boundedness in the space $XE$.\\

2. As already mentioned in the introduction, for Schr\"odinger operators $H=-\Delta+V$ with certain potentials $V$, a similar concept of generalized Triebel-Lizorkin spaces is developed in \cite{zheng1} and \cite{zheng2}, where also generalized Triebel-Lizorkin spaces associated to Schr\"odinger operators are defined and studied. The definition given there differs from our definition and is closer to the original definition of Triebel-Lizorkin spaces via a Littlewood-Paley like   decomposition. In particular, the fact that the considered Schr\"odinger operators are self-adjoint is essential, since the auxiliary functions used to define the spaces are in the class $C_c^\infty(\R)$. In contrary, our concept is more general in the way that we can handle also sectorial operators with non-real spectrum. Nevertheless, although we do not study this here, we are convinced that the generalized Triebel-Lizorkin spaces introduced in \cite{zheng1} for Schr\"odinger operators coincide with our notion of $s$-intermediate spaces, at least, if the negative part of the potential is in the Kato class. A proof could be based on suitable modifications of the methods in \cite{kriegler}, Chapter 4.4, where only the case $s=2$ is treated in connection with Littlewood-Paley decompositions.

\bibliography{Kunstmann-Ullmann-Literatur}{}

\bibliographystyle{plain}

\end{document}